\documentclass[12pt]{article}
\usepackage{amssymb}	
\usepackage{amsmath}    
\usepackage{epsfig}
\usepackage{multirow}
\usepackage{longtable}
\usepackage{algorithmic}

\newenvironment{proof}[1][Proof:]{\begin{trivlist} 
\item[\hskip \labelsep {\bfseries #1}]}{\end{trivlist}} 
\newcommand{\qed}{\nobreak \ifvmode \relax \else \ifdim\lastskip<1.5em \hskip-\lastskip \hskip1.5em plus0em minus0.5em \fi \nobreak \vrule height0.75em width0.5em depth0.25em\fi} 
\def\0{\bf \0}
\def\A{{\bf A}}
\def\B{{\bf B}}
\def\C{{\bf C}}
\def\D{{\bf D}}

\def\H{{\bf H}}
\def\I{{\bf I}}
\def\J{{\bf J}}

\def\L{{\bf L}}

\def\0{{\bf 0}}

\def\R{\mathbb{R}}
\def\S{{\bf S}}
\def\T{{\bf T}}

\def\W{{\bf W}}
\def\X{{\bf X}}

\def\Z{{\bf Z}}
\def\a{{\bf a}}
\def\b{{\bf b}}
\def\c{{\bf c}}
\def\d{{\bf d}}
\def\e{{\bf e}}

\def\g{{\bf g}}
\def\h{{\bf h}}

\def\k{{\bf k}}

\def\p{{\bf p}}

\def\s{{\bf s}}

\def\v{{\bf v}}
\def\w{{\bf w}}
\def\x{{\bf x}}
\def\y{{\bf y}}
\def\z{{\bf z}}

\def\T{{\rm T}}

\newtheorem{algorithm}{Algorithm}[section]

\newtheorem{theorem}{Theorem}[section]
\newtheorem{lemma}{Lemma}[section]
\newtheorem{proposition}{Proposition}[section]

\newtheorem{remark}{Remark}[section]

\usepackage{color}
\usepackage[normalem]{ulem}
\newcommand*{\AN}[1]{(#1)}

\begin{document}
\title{An arc-search interior-point algorithm for nonlinear constrained optimization}
\author{
Yaguang Yang\thanks{Goddard Space Flight Center, NASA, 
8800 Greenbelt Rd, Greenbelt, MD 20771. 
Email: yaguang.yang@nasa.gov.} 
}

\date{\today}

\maketitle    

\begin{abstract}

This paper proposes a new arc-search interior-point algorithm for the nonlinear 
constrained optimization problem. The proposed algorithm uses the second-order
derivatives to construct a search arc that approaches the optimizer. Because the 
arc stays in the interior set longer than any straight line, it is expected that the
scheme will generate a better new iterate than a line search method.
The computation of the second-order derivatives requires to solve the
second linear system of equations, but the coefficient matrix of the second
linear system of equations is the same as the first linear system of equations.
Therefore, the matrix decomposition obtained while solving the first linear system 
of equations can be reused. In addition, most elements of the right-hand side 
vector of the second linear system of equations are already computed when the 
coefficient matrix is assembled. Therefore, the computation cost for solving 
the second linear system of equations is insignificant and the benefit of having
a better search scheme is well justified. The convergence of the proposed 
algorithm is established. Some preliminary test results are reported to demonstrate
the merit of the proposed algorithm. 
\end{abstract}


 
\section{Introduction}

The interior-point method has been very successful in solving the large scale linear 
optimization problem \cite{mehrotra92,wright97}. It is then extended
to solve the convex quadratic optimization problem \cite{ma89}, the general
convex optimization problem \cite{Monteiro94}, the linear complementarity 
problem \cite{kmy89}, and the semi-definite optimization problem \cite{aliz91}.
Not surprisingly, there are many interior-point algorithms developed for 
the nonlinear optimization problem 
\cite{bs08,bhn99,bgn00,bmn04,curtis12,ettz96,fgw02,gr12,ld20,ls04,nww09,plantenga98,twbul98,uuv04,vs99,wb05,wb06,yiy22}.
Most of the proposed algorithms use the line search technique but the trust region technique 
is also considered by researchers, for example \cite{bhn99,bgn00,plantenga98}. 
For the linear optimization problem, the path-following interior-point method is most popular 
because of its efficiency, but for the nonlinear optimization problem, using logarithmic 
barrier functions is widely considered because of its clear connection to the penalty method, 
exceptions are \cite{ettz96,uuv04,yiy22}. Although using higher-order derivatives in 
interior-point method is regarded as a useful strategy for linear optimization problems 
\cite{mehrotra92,wright97}, it is not be adopted by almost all interior-point 
algorithms designed for nonlinear optimization problems as it was pointed out 
in \cite{yiy22} that the right-hand side vector of the 
second-order derivatives equations involves the computation of the third-order
derivatives of Lagrangian, which is very expensive.

In this paper, we propose a higher-order arc-search interior-point algorithm which is 
different from most existing algorithms in that (a) it is a path-following algorithm, its 
merit for linear optimization problems is demonstrated in \cite{wright97};
(b) it uses an arc instead of a straight line to search for an optimizer, a useful strategy 
that has been proved by different researchers for various optimization problems 
\cite{iy24,kheirfam17,kheirfam21,ylz17,yang11,yang13,yang18,yang20,zyzlh19,zhl21},
in particular, this author showed in \cite{yang18} that an arc-search infeasible-interior 
point algorithm using higher-order derivatives, searching in a widest neighborhood, 
and starting from an infeasible point, is computationally competitive to the famous 
Mehrotra's algorithm and has the best-known convergence rate, which solved 
a long-standing problem, puzzling the researchers for years \cite{todd02};
and (c) it solves two linear systems of equations with the same square coefficient matrix
using just one matrix decomposition, a brilliant idea introduced by Mehrotra 
\cite{mehrotra92} and used by interior-point algorithms designed for various 
optimization problems  \cite{kheirfam21,wright97,yang18}. 
The proposed algorithm improves the efficiency of the algorithm of \cite{yiy22} 
by avoiding the computation of the third-order derivatives of Lagrangian on 
the right-hand side of the second linear system of equations. Since
most elements of the reduced right-hand side of the second linear 
system of equations are obtained already while we construct the square coefficient 
matrix, the cost of solving the second linear system of equations is insignificant 
but we get a significantly better search arc. A warm-restart update scheme for the
relaxation variables and multipliers is introduced aiming at further improving the 
computational efficiency. The convergence of the proposed algorithm is proved under 
some mild conditions similar to the ones assumed in \cite{ettz96,yiy22}. 
Limited by the length of this paper, we provide only some preliminary test
results, and leave all implementation details and extensive test results in a follow-on 
paper. Some applications of the proposed algorithm to the spacecraft optimal
trajectory design problems  are reported in a recent conference paper \cite{yph24},
and another application to a spacecraft formation flying orbit optimization problem is 
discussed in \cite{yang25}.

The remainder of the paper is organized as follows.
Section~\ref{sec:description} introduces the nonlinear optimization problem.
Section~\ref{sec:algorithm} presents the proposed arc-search
interior-point algorithms. Section~\ref{sec:convergence} discusses
their convergence properties. Section \ref{sec:test} provides 
preliminary numerical test results on some benchmark test problems. 
Section~\ref{sec:conclusions} summarizes the conclusion of the paper. 

\section{Problem description}
\label{sec:description}

The following notations will be used throughout this paper. We use bold small
letters for vectors, bold capital letters for matrices, and normal font for scalars. 
For a vector $\x$, we use $x_i$ for the $i$-th element of $\x$, and $\X$ for a 
diagonal matrix whose diagonal elements are the vector $\x$. For vectors or
matrices, we will use the superscript $k$ to 
represent their values at the $k$-th iteration. For scalars, we use the subscript $k$
to represent their values at the $k$-th iteration, while using the superscript to 
represent their powers. Finally, we use $\R_+^n$ ($\R_{++}^n$)  to
denote the space of nonnegative vectors (positive vectors, respectively),
and use $\e$ to denote a vector of all ones with appropriate dimension.

In this paper, we consider the following nonlinear optimization problem.
\begin{align}
\begin{array}{rcl}
\min &:& f(\x) \\
\textrm{s.t.} &:& \h(\x) = \0, \label{NP}\\
& & \g(\x) \ge \0,
\end{array}
\end{align}
where $f: \R^n \rightarrow \R$,  $\h: \R^n \rightarrow \R^m$, and 
$\g: \R^n \rightarrow \R^p$. To ensure that the system has a solution, we assume
$n>m \ge 0$. Without loss of generality, we assume $p \ge 1$, that is needed
for interior-point method and simplifies the discussion. The decision vector is $\x \in \R^n$. 
We introduce a nonnegative relaxation vector $\s \ge \0$ to rewrite (\ref{NP}) as
\begin{align}
\begin{array}{rcl}
\min &:&  f(\x)  \\
\textrm{s.t.} &:&  \h(\x)=\0, \label{NP1} \\
& & \g(\x) - \s = \0, \\
& & \s \ge \0.
\end{array}
\end{align}
For problem (\ref{NP1}), the Lagrangian function is given by
\begin{equation}
L(\x,\y,\w,\s,\z)=f(\x)-\y^{\T}\h(\x)-\w^{\T}(\g(\x)-\s)-\z^{\T}\s,
\label{lagrangian1}
\end{equation}
where $\y \in \R^m, \w \in \R^p_+$, and $\z \in \R^p_+$  are multipliers.
To simplify the notation, we use $\v= (\x,\y,\w,\s,\z) \in \R^{n+m+3p}$ to 
denote a stacked column vector composed of decision variables, relaxation 
variables, and multipliers. 
The gradients of the Lagrangian function with respect to $\x$ and $\s$ are
given by
\begin{subequations}
\begin{gather}
\nabla_{\x} L(\v)=\nabla f(\x)-\nabla \h(\x) \y-\nabla \g(\x) \w,
\\
\nabla_{\s} L(\v)=\w-\z,
\end{gather}
\label{dlagrangian}
\end{subequations}
respectively, where $\nabla \h(\x) \in \R^{n \times m}$ and
$\nabla \g(\x) \in \R^{n \times p}$ are Jacobians given by
\begin{align*}
\nabla \h(\x)  = \left[ \begin{array}{cccc}
\frac{\partial  h_1}{\partial x_1 }  &
\frac{\partial  h_2}{\partial x_1 }  &
\cdots &
\frac{\partial  h_m}{\partial x_1 }  \\
\vdots  & \vdots  & \ddots  & \vdots \\
\frac{\partial  h_1}{\partial x_n }  &
\frac{\partial  h_2}{\partial x_n }  &
\cdots &
\frac{\partial  h_m}{\partial x_n } 
\end{array} \right] 
=[\nabla h_1(\x),   \cdots, \nabla  h_m(\x) ],
\end{align*}
and
\begin{align*}
\nabla \g(\x)  = \left[ \begin{array}{cccc}
\frac{\partial  g_1}{\partial x_1 }  &
\frac{\partial  g_2}{\partial x_1 }  &
\cdots &
\frac{\partial  g_p}{\partial x_1 }  \\
\vdots  & \vdots  & \ddots  & \vdots \\
\frac{\partial  g_1}{\partial x_n }  &
\frac{\partial  g_2}{\partial x_n }  &
\cdots &
\frac{\partial  g_p}{\partial x_n }
\end{array} \right] 
=[\nabla g_1(\x),   \cdots, \nabla  g_p(\x) ] .
\end{align*}
Then, the KKT conditions of (\ref{NP1}) are given by (see \cite{nw06}):
\begin{align}
\k(\v)  = \left[ \begin{array}{c}
\nabla_{\x} L(\v)  \\
\h(\x)  \\
\g(\x)-\s \\
\w-\z  \\
\Z\s
\end{array} \right] =  \0,  \
(\w, \s, \z) \in \R_+^{3p}, 
\label{KKT1}
\end{align}
where $\k:\R^{n+m+3p} \to \R^{n+m+3p}$. It is worthwhile to note that $\nabla_{\x} L(\v)$
depends on the choice of the multiplier vector $\y$ which has no restriction as long as its elements
are real numbers. In addition, the multiplier $\y$ is unrelated to any other components of $\k(\v)$.
Moreover, $\s > \0$ is unrelated to any components of $\k(\v)$ except $\g(\x)-\s=\0$.
These features will be used while we introduce a warm-restart scheme to minimize $\| \k(\v)  \|$. 
For the vector $\k(\v)$, its value at the $k$-th iteration is expressed by
\begin{align}
\k(\v^k)  = \left[ \begin{array}{c}
\nabla f(\x^k)-\nabla \h(\x^k) \y^k-\nabla \g(\x^k) \w^k  \\
\h(\x^k)  \\
\g(\x^k)-\s^k \\
\w^k-\z^k  \\
\Z^k \s^k
\end{array} \right].
\label{KKT2}
\end{align}
Therefore, the Jacobian of $\k(\v)$ is given by
\begin{align}
\k' (\v) = \left[ \begin{array}{ccccc}
\nabla_{\x}^2 L(\v) & -\nabla \h(\x) & -\nabla \g(\x) & \0 & \0 \\
\left( \nabla \h(\x)\right)^{\T} & \0 & \0 & \0 & \0  \\
\left( \nabla \g(\x)\right)^{\T} & \0 & \0 & -\I & \0 \\
\0  & \0  & \I  &  \0  &  -\I  \\
\0 & \0 & \0  &  \Z & \S
\end{array} \right], 
\label{firstJacobian}
\end{align}
where 
\begin{equation}
\nabla_{\x}^2 L(\v) = \nabla_{\x}^2 f(\x) 
- \sum_{i=1}^m \nabla_{\x}^2 h_i(\x)  y_i
- \sum_{i=1}^p \nabla_{\x}^2 g_i(\x)  w_i.
\label{2ndL}
\end{equation}
Let the index set of the active inequality constraints $\g(\x)$ at $\x \in \R^n$ 
be denoted by
\begin{equation}
I(\x)=\left\{i \in \{1,\ldots,p\} : g_i(\x)=0 \right\}.
\label{index}
\end{equation}
For convergence analysis, we need the definition of Lipschitz continuity. Let 
$\alpha(\x)$ and $\beta(\x)$ be real functions of $\x$. Let an open set
$\mathcal{M}$ contain the level set 
$\mathcal{L}=\{ \x~|~ \alpha(\x) \le \alpha(\x_0) \}$, we say $\beta(\x)$ is 
Lipschitz continuous if there is a constant $L>0$ such that 
\begin{equation}
\| \beta(\x)-\beta(\y) \| \le L \| \x - \y \|,
\label{lipschitz}
\end{equation}
for all $\x, \y \in \mathcal{M}$.
Under some mild assumptions used in \cite{ettz96,nw06,yiy22}, similar to the proof of 
\cite{yang22}, it is straightforward to show the following result.

\begin{proposition}
Assume existence, smoothness, regularity, sufficiency, and strict complementarity
for problem (\ref{NP1}) described as follows:
\begin{itemize}
\item[(A1)] Existence. There exists an optimal solution and its associate 
multipliers $\v^*=(\x^*,\y^*,\w^*,\s^*,\z^*)$ of (\ref{NP1}). 
Clearly, the KKT conditions (\ref{KKT1}) hold for any optimal solution.
\item[(A2)] Smoothness. $f(\x)$, $\h(\x)$ and $\g(\x)$ are differentiable up to 
the third order. In addition, $\nabla_{\x} L(\x)$, $\g(\x)$, and $\h(\x)$ 
are Lipschitz continuous.
\item[(A3)]  Regularity. The set $\{ \nabla h_j(\x^*) : j =1, \ldots, m\} \cup
\{ \nabla g_i(\x^*)  : i \in I(\x^*) \}$ is linearly independent.
\item[(A4)] Sufficiency. For all $\eta  \in \R^n \backslash \{0\}$,
we have $ \eta^{\T} \nabla_{\x}^2 L(\v^*)  \eta >0$.
\item[(A5)]  Strict complementarity. For each
$i \in \{ 1, \ldots, p \}$, we have $z_i^*+s_i^* >0$ and $z_i^* s_i^* =0$.
\end{itemize}
Then the Jacobian matrix at the optimal 
solution $\v^*$ is nonsingular, i.e., $[k' (\v^*)]^{-1}$ exists and is bounded.
\label{yang22}
\end{proposition}

\begin{remark}
Assuming the continuity of $[k' (\v^*)]^{-1}$, Proposition~\ref{yang22} implies that 
all local optimizers are isolated.
\end{remark}

\section{The arc-search interior-point algorithm}
\label{sec:algorithm}

The main idea of the proposed arc-search interior-point algorithm is to extend
a similar algorithm for linear programming presented in \cite{yang18} to the 
nonlinear programming problem. Given the current iterate  
${\v^k}=({\x^k},{\y^k},{\w^k},{\s^k},{\z^k})$ with $({\w^k},{\s^k},{\z^k})>\0$ 
and $t \ge 0$, we consider the solution $\v(t)=(\x({t}), \y({t}), {\w}({t}), {\s}({t}), {\z}({t}))$ 
of the following nonlinear system of equations defined for $t \in [0,1]$
\begin{eqnarray}
\left[ \begin{array}{l}
\nabla_{\x} L({\v}({t})) \\
\h({\x}({t})) \\
\g({\x}({t}))-\s({t}) \\
\nabla_{\s} L({\v}({t}))  \\
{\Z}({t}){\s}({t})
\end{array} \right]= t \cdot \left[ \begin{array}{l}
\nabla_{\x} L({\v^k}) \\
\h({\x^k}) \\
(\g({\x^k})-{\s^k}) \\
\nabla_s L({\v^k})  \\
{\Z^k} {\s^k} 
\end{array} \right],
\hspace{0.1in}
({\w}({t}), {\s}({t}), {\z}({t})) \ge \0.
\label{KKTcurve}
\end{eqnarray}
Clearly, when $t=1$, the system has a solution of $\v(1)=\v^k$; and 
when $t \rightarrow 0$, the system approaches the solution of the KKT system.
Under some mild conditions to be defined in (B1)-(B4), and using the inverse function
theorem \cite{luenberger69}, Yamashita et al. \cite{yiy22} showed the following result.
\begin{proposition}[\cite{yiy22}]
Let $\Omega(\epsilon)$ be defined in (\ref{epsilonSet}), assume that conditions 
(B1)-(B4) given in Section~\ref{sec:convergence} hold,
and $\v^k \subset \Omega(\epsilon)$, then the solution $\v(t)$ of the system of 
the nonlinear equations (\ref{KKTcurve}) is unique for every $t \in (0,1]$.
\label{yiy22}
\end{proposition}
Therefore, $\v(t) \in \R^{n}\times \R^{m}\times \R^{p}\times  \R^{p}\times \R^{p}$ 
is an arc that starts at $\v^k$ and guides the current iterate to a candidate of a local
optimal solution of (\ref{NP}) because (\ref{KKTcurve}) reduces to the KKT conditions
when $t \rightarrow 0$. Therefore, the proposed algorithm is a path-following algorithm.

However, solving the system of nonlinear equations of (\ref{KKTcurve}) is impractical.
Similar to the idea used in \cite{yiy22,yang18}, our strategy is to approximate the arc $\v(t)$
by part of an ellipse which is defined as:
\begin{equation}
{\cal E}=\lbrace \v (\alpha): 
\v (\alpha) =
\vec{\a}\cos(\alpha)+\vec{\b}\sin(\alpha)+\vec{\c}, \alpha \in [0, 2\pi] \rbrace,
\label{ellipse}
\end{equation}
where $\vec{\a} \in \R^{n+m+3p}$ and $\vec{\b} \in \R^{n+m+3p}$ are 
the axes of the ellipse, and $\vec{\c} \in \R^{n+m+3p}$ is its center. 
Then, we search the optimizer along the ellipse $\mathcal{E}$ to get a
new iterate $\v^{k+1}$ and repeat the process until an $\epsilon$-optimizer 
is found. To construct the ellipse of (\ref{ellipse}), we assume that
the first and second derivatives of (\ref{KKTcurve}) and (\ref{ellipse}) 
are the same at current iterate $\v^k$ because (\ref{ellipse}) is an estimate
of (\ref{KKTcurve}).

Taking the derivative on both sides of (\ref{KKTcurve}) yields
\begin{equation}
\left[ \begin{array}{ccccc}
\nabla_{\x}^2 L({\v^k}) & -\nabla \h({\x^k}) & -\nabla \g({\x^k}) & \0  & \0 \\
\left( \nabla \h({\x^k})\right)^{\T} & \0 & \0 & \0 & \0  \\
\left( \nabla \g({\x^k})\right)^{\T} & \0 & \0 & -\I  & \0 \\
\0  & \0  &  \I  &  \0  &  -\I  \\
\0 & \0  & \0 & \Z^k & \S^k
\end{array} \right]
\left[ \begin{array}{c}
\dot{\x} \\ \dot{\y} \\ \dot{\w} \\ \dot{\s}  \\ \dot{\z}
\end{array} \right]
= \left[ \begin{array}{c}
\nabla_{\x} L({\v^k}) \\
\h({\x^k}) \\
\g({\x^k})-{\s^k} \\
{\w^k}-{\z^k}  \\
\Z^k {\s^k} 
\end{array} \right].
\label{firstOrder}
\end{equation}

It has been shown in \cite{wright97} that the search performance may
be improved by adding a centering term in (\ref{firstOrder}). Let the dual measure 
be defined as:
\begin{equation}
{\mu}=\frac{{\z}^{\T}{\s}}{p}.
\label{dualM}
\end{equation}
Following the common strategy used in most interior point algorithms (see \cite{wright97}), 
we introduce a centering parameter $\sigma$ and modify (\ref{firstOrder}) as:
\begin{equation}
\left[ \begin{array}{ccccc}
\nabla_{\x}^2 L({\v^k}) & -\nabla \h({\x^k}) & -\nabla \g({\x^k}) & \0  & \0 \\
\left( \nabla \h({\x^k})\right)^{\T} & \0 & \0 & \0 & \0  \\
\left( \nabla \g({\x^k})\right)^{\T} & \0 & \0 & -\I  & \0 \\
\0  & \0  &  \I  &  \0  &  -\I  \\
\0 & \0  & \0 & \Z^k & \S^k
\end{array} \right]
\left[ \begin{array}{c}
\dot{\x} \\ \dot{\y} \\ \dot{\w} \\ \dot{\s}  \\ \dot{\z}
\end{array} \right]
= \left[ \begin{array}{c}
\nabla_{\x} L({\v^k}) \\
\h({\x^k}) \\
\g({\x^k})-{\s^k} \\
{\w^k}-{\z^k}  \\
\Z^k {\s^k}  -\sigma {\mu}_k \e
\end{array} \right].
\label{firstOrderM}
\end{equation}
where $\e$ is an all-one vector of dimension $p$. This modification enforces the arc $\v(\alpha)$
toward the interior of the constraints, thereby makes sure that a substantial segment of the 
ellipse satisfies $(\w(\alpha), \s(\alpha), \z(\alpha))>\0$, which may increase the step size and 
improve the efficiency of the algorithm. Since (\ref{firstOrder}) is a special case of (\ref{firstOrderM})
with $\sigma=0$, we will use (\ref{firstOrderM}) throughout the rest paper.

Since the right-hand side of (\ref{firstOrderM}) at $\v^k$ is a constant vector,
taking the derivative on both sides of (\ref{firstOrderM}) yields
\begin{eqnarray}
\left[ \begin{array}{ccccc}
\nabla_{\x}^2 L(\v^k) & -\nabla \h(\x^k) & -\nabla \g(\x^k) & \0  & \0 \\
\left( \nabla \h(\x^k)\right)^{\T} & \0 & \0 & \0 & \0  \\
\left( \nabla \g(\x^k)\right)^{\T} & \0 & \0 & -\I  & \0 \\
\0  & \0  &  \I  &  \0  &  -\I  \\
\0 & \0 & \0  & \Z^k & \S^k
\end{array} \right]
\left[ \begin{array}{c}
\ddot{\x} \\ \ddot{\y} \\ \ddot{\w} \\ \ddot{\s} \\ \ddot{\z}
\end{array} \right]
=  \left[ \begin{array}{c}
\p(\v^k) \\
-(\nabla_{\x}^2 \h(\x^k))^{\T} \dot{\x} \dot{\x} \\
-(\nabla_{\x}^2 \g(\x^k))^{\T} \dot{\x} \dot{\x} \\
\0  \\
-2\dot{\Z} \dot{\s} 
\end{array} \right],
\label{secondOrder}
\end{eqnarray}
where 
\[
\p(\v^k)=-(\nabla_{\x}^3 L(\v^k))\dot{\x} \dot{\x}
+2(\nabla_{\x}^2 \h(\x^k))\dot{\y} \dot{\x}
+2(\nabla_{\x}^2 \g(\x^k))\dot{\w} \dot{\x}.
\]
The formulas for computing the elements on the right-hand side of (\ref{secondOrder}) 
can be found in \cite[Appendix 1]{yiy22}, and we provide them below for completeness.

\begin{eqnarray}
\nabla_{\x}^3 L(\x,\y,\z)\dot{\x} \dot{\x} 
&=& 
\frac{\partial \left( \frac{\partial^2 L(\x,\y,\z)}{\partial \x^2} \dot{\x} \right) }
{\partial \x} \dot{\x} 
=  \sum_{i=1}^n  \dot{x}_i \frac{\partial}{\partial \x}
\left[ \begin{array}{c}
	\frac{\partial^2 L(\x,\y,\z)}{\partial x_1 \partial x_i} \\
	\vdots \\
	\frac{\partial^2 L(\x,\y,\z)}{\partial x_n \partial x_i}
\end{array} \right] \dot{\x} 
\label{Lxx}
\end{eqnarray}
\begin{eqnarray}
\nabla_{\x}^2 \h(\x)\dot{\y} \dot{\x} 
&=&
\frac{\partial \left( \frac{\partial  \h(\x)}{\partial \x } \dot{\y} \right) }
{\partial \x} \dot{\x} 
=  \sum_{i=1}^m \dot{y}_i \frac{\partial}{\partial \x}
\left[ \begin{array}{c}
	\frac{\partial  h_i(\x)}{\partial x_1 } \\
	\vdots \\
	\frac{\partial  h_i(\x)}{\partial x_n} 
\end{array} \right] \dot{\x}
=  \sum_{i=1}^m \dot{y}_i
\left( \nabla_{\x}^2 h_i(\x) \right)  \dot{\x} 
\label{hyx}
\end{eqnarray}
\begin{eqnarray}
\nabla_{\x}^2 \g(\x)\dot{\w} \dot{\x} 
&=&
\frac{\partial \left( \frac{\partial  \g(\x)}{\partial \x } \dot{\w} \right) }
{\partial \x} \dot{\x} 
=  \sum_{i=1}^n \dot{w}_i \frac{\partial}{\partial \x}
\left[ \begin{array}{c}
	\frac{\partial  g_i(\x)}{\partial x_1 } \\
	\vdots \\
	\frac{\partial g_i(\x)}{\partial x_n} 
\end{array} \right] \dot{\x}
= \sum_{i=1}^n \dot{w}_i
\left( \nabla_{\x}^2 g_i(\x) \right)  \dot{\x}  
\label{gzx}
\end{eqnarray}
\begin{eqnarray}
\nabla_{\x}^2 \h(\x)^{\T} \dot{\x}  \dot{\x}
&=&
\left( \frac{\partial \left( \left( \frac{\partial  \h(\x)}
	{\partial \x } \right)^{\T} \dot{\x} \right) }{\partial \x} \right)^{\T} \dot{\x}
=  
\left[ \begin{array}{c}
	\dot{\x}^{\T} \left( \nabla_{\x}^2 h_1(\x) \right)  \dot{\x} \\
	\vdots \\
	\dot{\x}^{\T} \left( \nabla_{\x}^2 h_m(\x) \right)  \dot{\x} 
\end{array} \right] 
\label{hxx}
\end{eqnarray}
\begin{eqnarray}
\nabla_{\x}^2 \g(\x)^{\T} \dot{\x}  \dot{\x}
&=&
\left( \frac{\partial \left( \left( \frac{\partial \g(\x)}
	{\partial \x } \right)^{\T} \dot{\x} \right) }{\partial \x} \right)^{\T} \dot{\x}
=  
\left[ \begin{array}{c}
	\dot{\x}^{\T} \left( \nabla_{\x}^2 g_1(\x) \right)  \dot{\x} \\
	\vdots \\
	\dot{\x}^{\T} \left( \nabla_{\x}^2 g_p(\x) \right)  \dot{\x} 
\end{array} \right].
\label{gxx}
\end{eqnarray}

\begin{remark}
It is noteworthy that the Hessian matrices of $\nabla_{\x}^2 h_i(\x)$ and
$\nabla_{\x}^2 g_i(\x)$ in (\ref{hyx}), (\ref{gzx}), (\ref{hxx}), and (\ref{gxx}) 
are obtained when $\nabla_{\x}^2 L(\v)$ is computed in (\ref{2ndL}). 
The major burden of the computation of the right-hand side
of (\ref{secondOrder}) is 
\begin{eqnarray}
& & \frac{\partial}{\partial \x} \left[ \begin{array}{c}
	\frac{\partial^2 L(\x,\y,\z)}{\partial x_1 \partial x_i} \\
	\vdots \\
	\frac{\partial^2 L(\x,\y,\z)}{\partial x_n \partial x_i}
\end{array} \right] 
= 
\sum_{i=1, \ldots, n} \frac{\partial^2}{\partial {\x}^2} 
\left[ 	\frac{\partial L(\x,\y,\z)}{\partial x_i} \right]  \label{extraBurden}  \\
&=&
\sum_{i=1, \ldots, n} \frac{\partial^2 }{\partial {\x}^2}\left( \frac{df({\x})}{dx_i} \right)
-\sum_{\tiny\begin{array}{c} i=1, \ldots, m \nonumber \\ j=1, \ldots, n \end{array}\normalsize} 
y_i \frac{\partial^2 }{\partial {\x}^2}\left( \frac{dh_i({\x})}{dx_j} \right)
-\sum_{\tiny\begin{array}{c} i=1, \ldots, p \nonumber \\ j=1, \ldots, n \end{array}\normalsize} 
w_i \frac{\partial^2 }{\partial {\x}^2}\left( \frac{dg_i({\x})}{dx_j} \right).
\end{eqnarray}
For quadratically constrained quadratic programming (QCQP) problem,
the increased computational burden of the right-hand side of 
(\ref{secondOrder}) is minimum because all terms in (\ref{extraBurden}) 
become zeros. 
\end{remark}

The computation of (\ref{extraBurden}) is expensive and adversely affects
the performance of the algorithm of \cite{yiy22}. Assuming the third-order
derivatives are relatively small (at least when iterates are close to the optimal solution),
we consider an alternative approximation by modifying (\ref{secondOrder}) to:
\footnotesize
\begin{eqnarray}
\left[ \begin{array}{ccccc}
\nabla_{\x}^2 L(\v^k) & -\nabla \h(\x^k) & -\nabla \g(\x^k) & \0  & \0 \\
\left( \nabla \h(\x^k)\right)^{\T} & \0 & \0 & \0 & \0  \\
\left( \nabla \g(\x^k)\right)^{\T} & \0 & \0 & -\I  & \0 \\
\0  & \0  &  \I  &  \0  &  -\I  \\
\0 & \0 & \0  & \Z^k & \S^k
\end{array} \right]
\left[ \begin{array}{c}
\ddot{\x} \\ \ddot{\y} \\ \ddot{\w} \\ \ddot{\s} \\ \ddot{\z}
\end{array} \right]
=  \left[ \begin{array}{c}
2[(\nabla_{\x}^2 \g(\x^k))\dot{\w} \dot{\x}+(\nabla_{\x}^2 \h(\x^k))\dot{\y} \dot{\x}] \\
-(\nabla_{\x}^2 \h(\x^k))^{\T} \dot{\x} \dot{\x} \\
-(\nabla_{\x}^2 \g(\x^k))^{\T} \dot{\x} \dot{\x} \\
\0  \\
-2\dot{\Z} \dot{\s} 
\end{array} \right],
\label{secondOrderM}
\end{eqnarray}
\normalsize
which ignore the third-order derivatives on the right-hand side, or even a simpler formula 
which ignore all high-order derivatives on the right-hand side, i.e.,
\begin{eqnarray}
\left[ \begin{array}{ccccc}
\nabla_{\x}^2 L(\v^k) & -\nabla \h(\x^k) & -\nabla \g(\x^k) & \0  & \0 \\
\left( \nabla \h(\x^k)\right)^{\T} & \0 & \0 & \0 & \0  \\
\left( \nabla \g(\x^k)\right)^{\T} & \0 & \0 & -\I  & \0 \\
\0  & \0  &  \I  &  \0  &  -\I  \\
\0 & \0 & \0  & \Z^k & \S^k
\end{array} \right]
\left[ \begin{array}{c}
\ddot{\x} \\ \ddot{\y} \\ \ddot{\w} \\ \ddot{\s} \\ \ddot{\z}
\end{array} \right]
=  \left[ \begin{array}{c}
\0 \\
\0 \\
\0 \\
\0  \\
-2\dot{\Z} \dot{\s} 
\end{array} \right].
\label{secondOrderM1}
\end{eqnarray}
Our preliminary numerical experience quickly denies the idea of using (\ref{secondOrderM1}) 
in the construction of the ellipse, because it oversimplifies the right-hand side. 
Although, using (\ref{secondOrderM1}) does reduce the computational cost
in every iteration, but it significantly increases iteration numbers, and overall is less
efficient than using (\ref{secondOrder}) or (\ref{secondOrderM}).

Let $\dot{\v}=(\dot{\x},\dot{\y},\dot{\w},\dot{\s},\dot{\z})$ be obtained by 
using (\ref{firstOrderM}), and $\ddot{\v}=(\ddot{\x},\ddot{\y},\ddot{\w},\ddot{\s},\ddot{\z})$ 
be obtained by using either (\ref{secondOrder}) or (\ref{secondOrderM}), which are 
the first-order and second-order derivatives of the ellipse $\mathcal{E}$ at the
current iterate. Then, the following theorem provides an equivalent formula of the 
ellipse (\ref{ellipse}).

\begin{theorem}\label{theorem:ellip}
\upshape{\cite{yang13}}
Suppose that the ellipse ${\cal E}$ of \textrm{(\ref{ellipse})} passes through 
the current iterate ${\v^k}$ at $\alpha=0$, and its first and second order derivatives at 
$\alpha=0$ are $\dot{\v}$ and $\ddot{\v}$, respectively. Then ${\cal E}$ is given
by ${\v}(\alpha) = ({\x}(\alpha), {\y}(\alpha), {\w}(\alpha), {\s}(\alpha), {\z}(\alpha))$ 
which can be calculated by
\begin{align}
{\v}(\alpha) = {\v^k} - \dot{\v}\sin(\alpha)+\ddot{\v}(1-\cos(\alpha)).
\label{vAlpha}
\end{align}
\label{arcFormula}
\end{theorem}
The proof of the theorem is given in \cite{yang13}. It is worthwhile to indicate that
$\v(t)$ starts from $t=1$ and ends at $t=0$ for the arc represented by (\ref{KKTcurve}),
but $\v(\alpha)$ starts from $\alpha=0$ and ends at $\alpha=\frac{\pi}{2}$ for the
arc represented by (\ref{ellipse}), i.e., $\v(t)|_{t=1} = \v(\alpha)|_{\alpha=0}=\v^k$
is the current iterate. Since the ellipse approximates the arc defined by 
(\ref{KKTcurve}), we assume that $\dot{\v}(t)|_{t=1} = \dot{\v}(\alpha)|_{\alpha=0}$ 
and $\ddot{\v}(t)|_{t=1} = \ddot{\v}(\alpha)|_{\alpha=0}$. For the sake of simplicity,
we will use $\dot{\v}$ for $\dot{\v}(\alpha)|_{\alpha=0}$ and 
$\ddot{\v}$ for $\ddot{\v}(\alpha)|_{\alpha=0}$.

It is noteworthy that the direction of 
$\dot{\v}=(\dot{\x},\dot{\y},\dot{\w},\dot{\s},\dot{\z})$ calculated by using
(\ref{firstOrder}) is the same one (with a different sign) that is used 
by the steepest descent method for solving (\ref{KKT1}). 
To see this, for a sequence of $\v^k$, we use $o(t_k)$ for the claim that the
condition $\| \v_k \| \le t_k$ holds. At iteration $k$, steepest descent method 
calculates the next iterate $\v^{k+1} = \v^{k} + \alpha_k \d^k$ 
by using the search direction $\d^k$ and searching an appropriate step 
size $\alpha_k$. Using Taylor's expansion for (\ref{KKT1}),  we have 
\begin{eqnarray}
\k(\v^{k+1}) & = & \k(\v^{k}) + \k'(\v^{k}) (\v^{k+1}-\v^{k}) 
+ o(\| \v^{k+1}-\v^{k} \|^2)  \nonumber \\
& = & \k(\v^{k}) + \k'(\v^{k}) (\alpha_k \d^k) 
+ o(\| \alpha_k \d^k \|^2)  
\end{eqnarray}
For steepest descent method, $\d^k$ is calculated by setting 
$\k(\v^{k}) + \k'(\v^{k})\d^k = \0$ (the same system of equations of (\ref{firstOrder})
with a different sign), i.e., 
\begin{equation}
\d^k = -[\k'(\v^{k})]^{-1}\k(\v^{k})= -\dot{\v},
\label{direction}
\end{equation}
which is essentially the same search direction of (\ref{vAlpha}) when the second-order
term $\ddot{\v}$ is not used.
That is the reason why the first-order interior-point algorithm is not efficient,
and the second-order correction interior-point algorithm is used by some of
the latest interior-point algorithms, such as IPOPT \cite{wb05} and \cite{yiy22}.
It is emphasized that the second-order correction by using the arc-search is
fundamentally different from the one in \cite{wb05}.

\begin{remark}
It is worthwhile to note that the direction $\d^k$ of (\ref{direction}) involves the 
inverse of the Jacobian matrix and for this reason many researchers working on linear optimization 
refer to the method as Newton's method. But $\d^k$ is not a conventional Newton 
direction because of the definition of the vector $\k(\v^{k})$ which is not the 
gradient of an objective function (see \cite[Page 259]{nw06}).
\end{remark}

\begin{remark}
The analytic arc-search formula provided by (\ref{vAlpha}) is 
simple and its computational cost is negligible compared to other ones
such as solving linear systems of equations. In addition, for boundary
constraints, no arc-search is required because one may use analytic
formulas provided in Appendix 2 of \cite{yiy22} to calculate the step size.
It is emphasized that all the analyses in the rest of the paper are applicable to both 
of the following two cases: (a) $\dot{\v}$ is calculated by (\ref{firstOrderM}) and 
$\ddot{\v}$ is calculated by (\ref{secondOrder}), and (b) $\dot{\v}$ is calculated by 
(\ref{firstOrderM}) and $\ddot{\v}$ is calculated by (\ref{secondOrderM}).
\end{remark}

Note that the fourth row of (\ref{firstOrderM}) indicates that 
if ${\w^k}={\z^k}$, then $\dot{\w}-\dot{\z}={\w^k}-{\z^k} = \0$. 
Similarly, in view of the fourth row of (\ref{secondOrder}) or (\ref{secondOrderM}), we have 
$\ddot{\w}-\ddot{\z}={\w^k}-{\z^k} = \0$. In view of Theorem \ref{arcFormula}, 
we have the following result.

\begin{lemma}[\cite{yiy22}]
If ${\v^k}$ satisfies ${\w^k}={\z^k}$, then ${\w}(\alpha)={\z}(\alpha)$ 
holds for any $\alpha \in [0,\frac{\pi}{2}]$.
\label{eqwz}
\end{lemma}

To search for the optimizer along the ellipse given by (\ref{vAlpha}), we expect that every
iterate will reduce the magnitude of $\k(\v^k)$ and eventually we will have 
$\k(\v^k) \rightarrow \0$ as $k \rightarrow 0$, thereby the KKT conditions 
will be satisfied. Therefore, we define the merit function by 
\begin{equation}
\phi(\v) = \| \k(\v) \|^2
\label{phiV}
\end{equation}
which should sufficiently decrease at ${\v}(\alpha)$ for 
some constants, $\rho \in (0, \frac{1}{2})$, 
$\sigma \in \left[0, 1 \right)$,  and step size $\alpha \in (0, \pi/2]$. 
The following lemma shows that this is achievable under the condition (\ref{cond1}).

\begin{lemma}
Let $\alpha \in (0, \frac{\pi}{2}]$, $\rho \in (0, \frac{1}{2})$ and 
$\sigma \in \left[0, 1 \right)$. Let $\mu = \frac{\z^{\T}\s}{p}$. If
\begin{equation}
\phi({\v}(\alpha)) \le \phi({\v}) + \rho \sin(\alpha)  \phi'({\v}(\alpha))|_{\alpha=0},
\label{cond1}
\end{equation}
then
\begin{equation}
\phi({\v}(\alpha)) \le \phi({\v}) (1 - 2\rho  (1-\sigma) \sin(\alpha))< \phi({\v}).
\label{conC1}
\end{equation}
\label{decrease}
\end{lemma}
\begin{proof}
Given the ranges of $\alpha, \rho$, and $\sigma$, the last inequality in (\ref{conC1}) 
clearly holds. The first inequality in (\ref{conC1}) follows from a similar argument used in 
\cite{ettz96}. Since $\dot{\v}$ is defined as the solution of (\ref{firstOrderM})
$\k'({\v}) \dot{\v}=\k({\v})-\sigma {\mu} \bar{\e}$,
where $\bar{\e}=(\0,\e)$ with $\0 \in \R^{n+m+2p}$ and $\e \in \R^{p}$,
using the definition (\ref{phiV}) and then (\ref{vAlpha}), we have
\begin{eqnarray}
\left. \frac{d \phi({\v}(\alpha))}{d \alpha}\right|_{\alpha=0} 
& = & \left. 2\k({\v(\alpha)})^{\T}\k'({\v(\alpha)}) \frac{d {\v}(\alpha)}{d \alpha}\right|_{\alpha=0} 
\nonumber \\
& = &  \left. 2\k({\v(\alpha)})^{\T}\k'({\v(\alpha)}) \left[
-\dot{\v} \cos(\alpha) + \ddot{\v}\sin(\alpha) \right]\right|_{\alpha=0}
\nonumber \\
& = & -2 \k({\v})^{\T}\k'({\v})\dot{\v} 
\label{tmp00} \\
& = &  -2 \k({\v})^{\T} [\k({\v}) -\sigma {\mu} \bar{\e}]
\nonumber \\
& = & -2 \phi({\v}) +2 \sigma {\mu}^2/p,
\label{tmp0}
\end{eqnarray}
where the last equality is derived from  
$\k({\v})^{\T} (\sigma {\mu} \bar{\e}) = \sigma {\mu} 
\sum_{i=1}^{p} {z}_i {s}_i = \sigma {\mu} {\z}^{\T}{\s}
=\sigma {\mu}^2 /p$. 
Since $|{\z}^{\T}{\s}| \le \sqrt{p} \| \Z {\s} \|_2$ and $p \ge 1$, 
we have
\[
{\mu}^2/p =({\z}^{\T}{\s})^2 /p \cdot (1/p^2) 
\le \|\Z{\s} \|_2^2 \cdot 1 \le \| \k({\v}) \|_2^2 = \phi({\v}).
\]
Substituting this inequality into (\ref{tmp0}), we have 
\begin{equation}
\phi'({\v}(\alpha))|_{\alpha=0} \le -2 \phi({\v}) +2 \sigma \phi({\v})
\le -2\phi({\v}) (1-\sigma).
\label{objReduction}
\end{equation}
From (\ref{cond1}), it holds 
\[
\phi({\v}(\alpha)) \le \phi({\v}) - 2 \rho \sin(\alpha) \phi({\v}) (1-\sigma)
=\phi({\v}) (1 - 2\rho  (1-\sigma) \sin(\alpha) ).
\]
This completes the proof.
\hfill \qed
\end{proof}

\begin{remark}
If Condition (\ref{cond1}) and $\sigma<\phi(\v) p/\mu^2$ hold, then 
$\phi(\v(\alpha))$ decreases because $\alpha, \rho>0$ and according to 
(\ref{tmp0}), $\phi'({\v}(\alpha))|_{\alpha=0} <0$.
\label{newL}
\end{remark}

As a matter of fact, Condition (\ref{cond1}) holds for $\sigma=0$.
In view of (\ref{tmp00}) and (\ref{firstOrder}), since 
$\left. \frac{d \phi({\v}(\alpha))}{d \alpha}\right|_{\alpha=0}
=-2 \k({\v})^{\T}\k'({\v})\dot{\v}$, it follows
\begin{eqnarray}
\phi({\v}(\alpha))
 & =& \phi({\v(\alpha)}) + \phi'({\v}) \alpha + o(\alpha^2)
\nonumber \\
 & =& \phi({\v(\alpha)}) - 2 \k({\v})^{\T}\k'({\v})\dot{\v}  \alpha + o(\alpha^2)
\nonumber \\
 & =& \phi({\v}) - 2\phi({\v}) \alpha + o(\alpha^2),
\end{eqnarray}
the last equation holds because of $\sigma=0$.
Therefore, for $\sigma=0$ and $\alpha >0$ small enough, we have 
$\phi({\v}(\alpha))<\phi({\v})$.
We summarize this result as the following lemma.
\begin{lemma}
Assume that $\sigma = 0$, then searching along the
ellipse using (\ref{vAlpha}) will always reduce the merit function.
\label{decrease1}
\end{lemma}

\begin{remark}
As $\phi({\v}(\alpha)) \rightarrow 0$, we find a KKT point which may be a maximizer. In a
follow-on paper, we will provide some implementation tricks to avoid this case.
\end{remark}

For any interior-point method, every iterate must be an interior-point, i.e., we
must have $(\w,\s,\z)>0$ and $\g(\x) > \0$ in all the iterations before the algorithm 
stops. In addition, all iterates should improve the merit function. To make sure that 
the algorithm is well-posed, we need that all iterates are bounded. In summary, 
we should choose the step size $\alpha \in (0, \frac{\pi}{2}]$
such that the following conditions hold.
\begin{itemize}
\item[(C1)] $(\w^k(\alpha_k)),\s^k(\alpha_k),\z^k(\alpha_k)) > \0$.
\item[(C2)] $\g(\x^k(\alpha_k)) > \0$.
\item[(C3)] $\phi(\v^{k+1}) = \phi(\v^k(\alpha_k)) < \phi(\v^k)$.
\item[(C4)] The generated sequence $\{\v^k\}$ should be bounded. 
\end{itemize}

Condition (C1) can be achieved by a process discussed in \cite{yang17,yang18}. 
Because of Lemma~\ref{eqwz}, we always have $\z^k(\alpha)=\w^k(\alpha)$. 
Therefore, we do not need to calculate both $\w^k(\alpha)$ and $\z^k(\alpha)$.
For a fixed a small $\delta_1 \in (0,1)$, we will select the largest 
$\tilde{\alpha}_k$ such that all $\alpha \in [0, \tilde{\alpha}_k]$ satisfy 
\begin{equation}\label{positive}
\w^{k+1}=\w^k(\alpha) =\w^k - \dot{\w}^k\sin(\alpha)+\ddot{\w}^k(1-\cos(\alpha)) 
\ge \delta_1 \w^k,  
\end{equation}
To this end, for each $i \in \lbrace 1,\ldots, p \rbrace$, we select the largest
$\alpha_{w_i}^k$ such that the $i$th inequality of (\ref{positive}) holds
for all $\alpha \in [0, \alpha_{w_i}^k ]$. Then, we define 
\begin{equation}
\tilde{\alpha}_k=\min_{i \in \lbrace 1,\ldots, p \rbrace}
\left\{ \min \left\{\alpha_{w_i}^k, \frac{\pi}{2} \right\} \right\}.
\label{alpha}
\end{equation}
The largest $\alpha_{w_i}^k$ can be computed in analytical forms 
(see the proof in \cite{yang17}), which are provided in Appendix 2 of \cite{yiy22}.
A lower bound estimation of $\tilde{\alpha}_k$ will be provided in Lemma~\ref{tildeAlpha}.
We may select $\s^{k+1}$ using the same process as we did for $\w^{k+1}$, 
but a warm start \cite{yw02} is better because it further reduces the merit function 
$\phi(\v^{k+1})=\phi(\v(\alpha))$. We will discuss this shortly. 

To meet Condition (C2), we search for $\bar{\alpha}_k \le \tilde{\alpha}_k$ 
for all $\alpha \in (0, \bar{\alpha}_k]$ and some $\delta_1 \in (0,1)$
such that $\g(\x^k(\alpha)) \ge \delta_1 \s^k > \0$. This can be achieved 
if $\g(\x^k) > \0$. Therefore, for all $\alpha \in (0, \bar{\alpha}_k]$, 
Condition (C2) holds.

To make sure that Condition (C3) holds, we search for 
$\check{\alpha}_k \in (0,\bar{\alpha}_k]$ such that condition (\ref{cond1}) holds, i.e., 
\begin{equation}
\check{\alpha}_k = \max\left\{ \alpha \in \left(0, \bar{\alpha}_k \right] : 
\phi({\v^k}(\alpha)) < \phi({\v^k}) - \delta_2 \right\}
\label{alphaCheck}
\end{equation}
where $\delta_2 \ge 0$ is a constant.
The existence of $\check{\alpha}_k$ is guaranteed due to Lemma~\ref{decrease1}. 
According to Remark~\ref{newL}, $\sigma_k <\phi(\v^k) p/\mu^2$ should be 
selected. A lower bound estimation for $\check{\alpha}_k$ will be provided in 
Lemma~\ref{checkAlpha}.

Finally, we will show that Condition (C4) holds under some mild conditions.
We will prove in Lemma \ref{boundedness} 
that $\{ \v^k \}$ is bounded, and $\{(\w^k, \s^k, \z^k)\}$ is bounded below 
and away from zero\footnote{A sequence $\{\c^k\} \subset \R^n $ is said to 
be \textit{bounded below and away from zero if there exists $\bar{c} > 0 $ 
such that every element of $\c^k$ satisfies $c_i^k \ge \bar{c}$
for all $k \ge 1$.}}. To this end, we would like to have $\v^k$ stay in
a desired set before our algorithms converge.
Let $\epsilon > 0$ be the convergence criterion
selected in the algorithms to be discussed, 
we define the desired set $\Omega(\epsilon)$ as follows:
\begin{equation}
\Omega (\epsilon) = \left\{ \v \in \R^{n+m+3p}: \epsilon \le \phi(\v) \le \phi(\v^0), 
\ \min (\Z\s) \ge \frac{1}{2} \frac{\Z^0\s^0}{\phi (\v^0)} \min (\phi (\v)) 
\right\}.
\label{epsilonSet}
\end{equation} 
To make sure that $\v^k$ stays in $\Omega(\epsilon)$, let
\begin{equation}
\hat{m}_k(\alpha)  = \min (\Z^k(\alpha) \s^k(\alpha) )- \frac{1}{2} 
\frac{\Z^0\s^0}{\phi (\v^0)} \min (\phi (\v(\alpha) )).
\label{measurePos}
\end{equation}
An $\hat{\alpha}_k$ is selected to satisfy $\hat{m}_k(\alpha_k) \ge 0$ for 
$\forall \alpha_k \in (0, \hat{\alpha}_k]$, i.e.,
\begin{equation}
\hat{\alpha}_k = \max\left\{ \alpha \in \left(0, \check{\alpha}_k \right] : \hat{m}_k(\alpha) \ge 0 \right\}.
\label{hatAlpha}
\end{equation}
This requirement has a similar effect as the wide neighborhood of interior-point 
methods \cite{wright97}. We will prove in Lemma \ref{AlphaH} that $\hat{\alpha}_k>0$ 
is bounded below and away from zero before our algorithms converge.

Summarizing the above process, we may select the step size in the $k$th iteration as:
\begin{equation}
\alpha_k = \min\{  \tilde{\alpha}, \check{\alpha}_k, \bar{\alpha}_k, \hat{\alpha}_k\}
=\hat{\alpha} > 0.
\label{alphaK}
\end{equation}
A reasonable update is $(\x^{k+1},\y^{k+1},\w^{k+1},\z^{k+1})
=(\x(\alpha_k),\y(\alpha_k),\w(\alpha_k),\z(\alpha_k))$.
As we mentioned above, our intuition and numerical experience made us believe that 
a warm start strategy will be more efficient because it further reduces the merit 
function in every iteration. Given $\x^{k+1}$, if we select 
$\s^{k+1}=\g(\x^{k+1})>0$, then, $\g(\x^{k+1})-\s^{k+1}= \0$, 
this reduces the third component of $\k(\v^{k+1})$ to zero and 
therefore reduces the merit function $\phi(\v^{k+1})$. We may also select 
$\y^{k+1}$ to reduce the first component of $\k(\v^{k+1})$ by minimizing
$\| \nabla f(\x^{k+1})-\nabla \h(\x^{k+1}) \y^{k+1}-\nabla \g(\x^{k+1}) \w^{k+1} \|$
because there is no restriction on $\y^{k+1}$ except that $\y^{k+1}$ is a real vector.
To achieve this, we can simply solve the following linear system of equations
\begin{equation}
\nabla \h(\x^{k+1}) \y^{k+1} = \nabla f(\x^{k+1})-\nabla \g(\x^{k+1}) \w^{k+1}
\label{ykp1}
\end{equation}
for $\y^{k+1}$.

We will compare algorithms proposed in this paper to the algorithm of Yamashita et. al \cite{yiy22} 
which is given as follows:

\begin{algorithm}{\bf (The first infeasible arc-search interior-point algorithm \cite{yiy22})} \\
\begin{algorithmic}[1]
\STATE Parameters: $\epsilon>0$, $\delta_1>0$, $\delta_2>0$,  
$\rho \in (0, \frac{1}{2})$, $\bar{\sigma} \in (0, \frac{1}{2})$, and $\gamma_{-1} \in [0.5,1)$.
\STATE Initial point: $\v^0 = (\x^0,\y^0,\w^0,\s^0,\z^0)$ 
such that $(\w^0,\s^0,\z^0) > \0$, $\g(\x^0) >0$,  and $\w^0=\z^0$.
\FOR{  k=0,1,2,\ldots}
	\STATE  Calculate $\nabla_{\x} L(\v^k)$, $\h(\x^k)$, $\g(\x^k)$, 
     $\nabla_{\x} \h(\x^k)$, $\nabla_{\x} \g(\x^k)$. If $\phi(\v^k) \le \epsilon$, stop.
	\STATE Calculate 
     $\nabla_{\x}^2 \h(\x^k)$, $\nabla_{\x}^2 \g(\x^k)$, and $\nabla_{\x}^2 L(\v^k)$.
	\STATE Select $\sigma_k$ such that 
	$\bar{\sigma} \le \sigma_k <  \min \{ \frac{1}{2}, \phi(\v^k) p/\mu^2 \}$.
	\STATE Calculate $\dot{\v}^k$ by solving (\ref{firstOrderM}) at ${\v} = \v^k$.
	\STATE Calculate the righthand side of (\ref{secondOrder}).
	\STATE Calculate $\ddot{\v}^k$ by solving (\ref{secondOrder})  at ${\v} = \v^k$.
	\STATE Calculate $\tilde{\alpha}_k$ using (\ref{alpha}) and search $\bar{\alpha}_k \in (0, \tilde{\alpha}_k]$
     such that $\g(\x(\bar{\alpha}_k)) > \0$.
	\STATE Search $\check{\alpha}_k \in (0, \bar{\alpha}_k]$ such that (\ref{alphaCheck}) holds.
	\STATE  Determine $\hat{\alpha}_k > 0$ using (\ref{measurePos}) and    
     (\ref{hatAlpha}).
     \STATE  Set $\alpha_k = \min \{ \hat{\alpha}_k, \check{\alpha}_k \}$.
	\STATE  Update $\v^{k+1} = \v^k(\alpha_k) = \v^k 
	- \dot{\v}^k \sin(\alpha_k) + \ddot{\v}^k (1-\cos(\alpha_k))$ .
	\STATE $k+1 \rightarrow k$.
\ENDFOR 
\end{algorithmic}
\label{mainAlgo0}
\end{algorithm}

Our first algorithm uses the accurate righthand side of (\ref{secondOrder}).

\begin{algorithm}{\bf (The second infeasible arc-search interior-point algorithm)} \\
\begin{algorithmic}[1]
\STATE Parameters: $\epsilon>0$, $\delta_1>0$, $\delta_2>0$,  
$\rho \in (0, \frac{1}{2})$, and $\bar{\sigma} \in [0, \frac{1}{2})$.
\STATE Initial point: $\v^0 = (\x^0,\y^0,\w^0,\s^0,\z^0)$ 
such that $(\w^0,\s^0,\z^0) > \0$, $\g(\x^0) >0$, and $\w^0=\z^0$.
\FOR{  k=0,1,2,\ldots}
	\STATE  Calculate $\nabla_{\x} L(\v^k)$, $\h(\x^k)$, $\g(\x^k)$, 
     $\nabla_{\x} \h(\x^k)$, $\nabla_{\x} \g(\x^k)$. If $\phi(\v^k) \le \epsilon$, stop.
	\STATE Calculate 
     $\nabla_{\x}^2 \h(\x^k)$, $\nabla_{\x}^2 \g(\x^k)$, and $\nabla_{\x}^2 L(\v^k)$.
	\STATE Select $\sigma_k$ such that 
	$\bar{\sigma} \le \sigma_k < \min \{ \frac{1}{8}, \phi(\v^k) p/\mu^2 \}$.
	\STATE Calculate  $\dot{\v}^k$ by solving (\ref{firstOrderM}) at ${\v} = \v^k$.
	\STATE Calculate the righthand side of (\ref{secondOrder}).
	\STATE Calculate $\ddot{\v}^k$ by solving (\ref{secondOrder})  at ${\v} = \v^k$.
	\STATE Calculate $\tilde{\alpha}_k$ using (\ref{alpha})  and search $\bar{\alpha}_k  \in (0, \tilde{\alpha}_k]$
     such that $\g(\x(\bar{\alpha}_k)) > \0$.
	\STATE Search $\check{\alpha}_k \in (0, \bar{\alpha}_k]$ such that (\ref{alphaCheck}) holds.
     \STATE  Determine $\hat{\alpha}_k > 0$ using (\ref{measurePos}) and    
     (\ref{hatAlpha}).
     \STATE  Set $\alpha_k = \min \{ \hat{\alpha}_k, \check{\alpha}_k \}$.
	\STATE  Update $(\x^{k+1},\w^{k+1}) =(\x^k,\w^k)
	- (\dot{\x}^k,\dot{\w}^k) \sin(\alpha_k) 
     + (\ddot{\x}^k,\ddot{\w}^k) (1-\cos(\alpha_k)).$
	\STATE  Select $\y^{k+1}$ by solving (\ref{ykp1}),
     and set $\s^{k+1} = \g(\x^{k+1}) >\0$ and $\z^{k+1}=\w^{k+1}$.
	\STATE $k+1 \rightarrow k$.
\ENDFOR 
\end{algorithmic}
\label{mainAlgo2}
\end{algorithm}

\begin{remark}
A major difference between Algorithms \ref{mainAlgo0} and \ref{mainAlgo2} 
is to use a warm start in Step 15 of Algorithm \ref{mainAlgo2}, where 
we select $\y^{k+1}$ to minimize
$\| \nabla f(\x^{k+1})-\nabla \h(\x^{k+1}) \y^{k+1}-\nabla \g(\x^{k+1}) \w^{k+1} \|$ 
and enforce equation $\g(\x^{k+1}) - \s^{k+1} = \0$ in every iteration $k$, 
which makes Algorithm \ref{mainAlgo2} converge faster and more robust.
\label{diffOldNew1}
%
Finding an appropriate initial point satisfying $\g(\x^0) >0$ is an important step 
and was not discussed in \cite{yiy22}, but will be discussed in a follow-on paper.
\end{remark}

Our extensive experience shows (i) the most expensive cost in Algorithms \ref{mainAlgo0}
and \ref{mainAlgo2} is the computation of the third order derivatives on the right
hand side of (\ref{secondOrder}), and (ii) when the iterates are close
to the optimal solution, the magnitude of the third-order derivative is small.
Therefore, we may ignore the third-order derivative terms in Algorithm \ref{mainAlgo2}.
This leads to an alternative infeasible arc-search interior-point algorithm that calculates 
$\dot{\v}$ and $\ddot{\v}$ by (\ref{firstOrderM}) and (\ref{secondOrderM}).

\begin{algorithm}{\bf (The third infeasible arc-search interior-point algorithm)} \\
\begin{algorithmic}[1]
\STATE Parameters: $\epsilon>0$, $\delta_1>0$, $\delta_2>0$,  
$\rho \in (0, \frac{1}{2})$, and $\bar{\sigma} \in [0, \frac{1}{2})$.
\STATE Initial point: $\v^0 = (\x^0,\y^0,\w^0,\s^0,\z^0)$ 
such that $(\w^0,\s^0,\z^0) > \0$, $\g(\x^0) >0$,  and $\w^0=\z^0$.
\FOR{  k=0,1,2,\ldots}
	\STATE  Calculate $\nabla_{\x} L(\v^k)$, $\h(\x^k)$, $\g(\x^k)$, 
     $\nabla_{\x} \h(\x^k)$, $\nabla_{\x} \g(\x^k)$. If $\phi(\v^k) \le \epsilon$, stop.
	\STATE Calculate 
     $\nabla_{\x}^2 \h(\x^k)$, $\nabla_{\x}^2 \g(\x^k)$, and $\nabla_{\x}^2 L(\v^k)$.
	\STATE Select $\sigma_k$ such that 
	$\bar{\sigma} \le \sigma_k <  \min \{ \frac{1}{8}, \phi(\v^k) p/\mu^2 \}$.
	\STATE Calculate $\dot{\v}^k$ by solving (\ref{firstOrderM}) at ${\v} = \v^k$.
	\STATE Calculate the righthand side of (\ref{secondOrderM}).
	\STATE Calculate $\ddot{\v}^k$ by solving (\ref{secondOrderM})  at ${\v} = \v^k$.
	\STATE Calculate $\tilde{\alpha}_k$ using (\ref{alpha}) and search $\bar{\alpha}_k \in (0, \tilde{\alpha}_k]$
     such that $\g(\x(\bar{\alpha}_k)) > \0$.
	\STATE Search $\check{\alpha}_k \in (0, \bar{\alpha}_k]$ such that (\ref{alphaCheck}) holds.
     \STATE  Determine $\hat{\alpha}_k > 0$ using (\ref{measurePos}) and    
     (\ref{hatAlpha}).
     \STATE  Set $\alpha_k = \min \{ \hat{\alpha}_k, \check{\alpha}_k \}$.
     \STATE  Update $(\x^{k+1},\w^{k+1}) =  (\x^k,\w^k)
	- (\dot{\x}^k,\dot{\w}^k) \sin(\alpha_k) 
     + (\ddot{\x}^k,\ddot{\w}^k) (1-\cos(\alpha_k)).$
	\STATE  Select $\y^{k+1}$ by solving (\ref{ykp1}),
     and set $\s^{k+1} = \g(\x^{k+1}) >\0$ and $\z^{k+1}=\w^{k+1}$.
	\STATE $k+1 \rightarrow k$.
\ENDFOR 
\end{algorithmic}
\label{mainAlgo1}
\end{algorithm}

\begin{remark}
By avoiding the calculation of the third-order derivatives on the
right-hand side of (\ref{secondOrder}), 
the major cost in the computation of the right-hand side of 
(\ref{secondOrderM}) are Hessians and they have been calculated 
when we construct the matrix of (\ref{firstOrderM}). Since the matrix decomposition 
of (\ref{secondOrderM}) is obtained when we solve (\ref{firstOrderM}),
this decomposition can be reused when we solve (\ref{secondOrderM}). Therefore, the 
computational cost of the higher-order Algorithm \ref{mainAlgo1} 
in every iteration is comparable to many first-order interior-point Algorithms such as 
\cite{bhn99,bgn00,ettz96,nww09,plantenga98,twbul98,uuv04,vs99}. 
With a careful implementation, our experience shows that the convergence of the 
Algorithm \ref{mainAlgo1} is faster than Algorithm \ref{mainAlgo2}. All implementation
details for this algorithm will be discussed in a follow-on paper.
\label{diffOldNew}
\end{remark}


\section{Convergence analysis}\label{sec:convergence}

Our convergence analysis is very general in the sense that it is applicable to all
three algorithms presented in the previous section. 
We will use the following assumptions like the ones used in \cite{ettz96,yiy22}.

\vspace{0.05in}
\noindent
{\bf Assumptions}
\begin{itemize}
\item[(B1)] The sequence $\{ \x^k \} \subset \R^n$ is bounded.
\item[(B2)]  For $\v \in \Omega(\epsilon)$, let $I_s^k$ be the index set 
$\{ i: 1 \le i \le p, \,\,  s^k_i =0 \}$. 
Then, the determinant of $(\J^k)^{\T} \J^k$ 
is bounded below and away from zero, where  $\J^k$ is a matrix 
whose column vectors are composed of
\[
\{ \nabla h_j (\x^k) : j = 1, \ldots, m\} \cup\{\nabla g_i(\x^k): i \in I_s^k  \}.
\]
\end{itemize}

\begin{remark}
Assumption (B2) is the standard LICQ condition \cite[page 328]{nw06}, which makes 
sure that the linear systems of equations (\ref{firstOrderM}), (\ref{secondOrder}), 
and (\ref{secondOrderM}) have solutions and helps us to establish
a strong convergence result, i.e., the algorithm converges in {\it finite steps}.
It is worthwhile to note that (B2) implies that in the set $\Omega(\epsilon)$, the 
columns of $\nabla \h(\x)$ are linearly independent.
\label{assumption0}
\end{remark}

The idea of the convergence analysis is similar to but not exactly the same as the one 
of \cite{yiy22}. We provide all proofs because it will be easy for
readers to check that the claims hold for all three algorithms. First, we establish 
the boundness results for series of functions and vectors.

\begin{lemma}
Assume that $\phi(\v^0)$ is bounded. Then
the series $\{ \phi(\v^k) \}$ created by any of the three Algorithms is bounded.
Therefore, the series $\{ \k(\v^k) \}$ is bounded. This implies that $\{ \nabla_{\x} L(\v^k) \}$,
$\{ \h(\x^k) \}$, $\{\g(\x^k)-\s^k  \}$, and $\{ \Z^k\s^k \}$ are all bounded.
\label{kvBound}
\end{lemma}
\begin{proof}
In view of Lemmas \ref{decrease} and \ref{decrease1}, it follows that
for all three Algorithms, $\{ \phi(\v^k)  \}$ is monotonically decreasing. 
This shows the first claim.
By the definition of $\{ \phi(\v^k) \}$, the second claim is true. The last
claim follows from the definition of $\k(\v^k)$ in (\ref{KKT1}).
\hfill \qed
\end{proof}

In the following few lemmas, we will show that $\{ \v^k \}$ is bounded.

\begin{lemma}
Assume that (B1) holds, then $\{ \s^k \} >0$ is bounded.
\label{sBound}
\end{lemma}
\begin{proof}
From (B1) and the continuity of $\g$, the boundedness of $\{ \x^k \}$ 
implies that $\{ \g(\x^k) \}$ is bounded because $\g$ is 
differentiable and therefore continuous. In view of Lemma \ref{kvBound}, 
since $\| \s^k \| \le \| \g(\x^k) -\s^k \| + \| \g(\x^k) \|$, it follows that 
$\{\s^k\}$ is bounded. The selection process of $\s^k$ guarantees that 
$\{ \s^k \} >0$ holds.
\hfill \qed
\end{proof}

\begin{lemma}
Assume that (B1) and (B2) hold and $\{ \z^0 \}=\{ \w^0 \}>0$, then $\{ \z^k \}=\{ \w^k \}>0$ 
is bounded.
\label{zBound}
\end{lemma}
\begin{proof}
In view of Lemma~\ref{eqwz}, we have $\w^k = \z^k$.
Since $\nabla f$ is continuous,  (B1) implies that $\{ \nabla f(\x^k) \}$ is bounded.
From  Lemma \ref{kvBound}, $\{\nabla_{\x} L(\v^k)\}$ is bounded. In view of (\ref{dlagrangian}),
it follows
\[
\| \nabla \h(\x^k) \y^k+\nabla \g(\x^k) \w^k \|
\le \| \nabla_{\x} L(\v^k) \| + \| \nabla f(\x^k) \|,
\]
therefore, $\{ \nabla \h(\x^k) \y^k+\nabla \g(\x^k) \w^k \}$ is bounded. 
Let $(\hat{\y}, \hat{\w} )=({\y}, {\w}) / \| ({\y}, {\w}) \|$.
Suppose by contradiction that  $z_i^k  =w_i^k  \rightarrow \infty$ 
as $k \to \infty$ for some $i$ and we denote this set as $I_s$. For
$z_j^k \notin I_s$, $z_j^k  =w_j^k  < \infty$ as $k \to \infty$. 
Since $\| (\hat{\y}, \hat{\w} ) \|=1$ and $\| ({\y}, {\w}) \| \rightarrow \infty$, 
$\hat{w}_i \neq 0$ if $i \in I_s$ and $\hat{w}_j = 0$ if $j \notin I_s$.
Note that $\{ \Z^k\s^k \} = \{ \W^k\s^k \}$ is bounded,
if $w_i^k \rightarrow \infty$, then $s_i^k \rightarrow 0$, this shows that 
$I_s$ is the same set as defined in (B2).
Since $\{ \nabla \h(\x^k) \y^k+\nabla \g(\x^k) \w^k \}$ is bounded
and $\hat{w}_j = 0$ for $j \notin I_s$, it follows 
\begin{equation}
\nabla \h(\x^k) \hat{\y}^k+\nabla \g(\x^k) \hat{\w}^k 
=\nabla \h(\x^k) \hat{\y} + \sum_{i \in I_s} \nabla g_i(\x^k) \hat{w}_i 
\rightarrow 0.
\label{contradict}
\end{equation}
This contradicts to the assumption of (B2). Therefore, $\{ \z^k \}=\{ \w^k \}$ 
is bounded. The selection process of $(\w^k, \z^k)$ gurantees that 
$(\w^k, \z^k) >0$ holds.
\hfill \qed
\end{proof}

\begin{lemma}
Assume that (B1) and (B2) hold, and $\{\v^k\} \subset \Omega(\epsilon)$, then  
$\{ s_i^k \}$ and $\{ z_i^k \}$ is bounded below from zero, and $\{ \y^k \}$ 
is bounded.
\label{yBound}
\end{lemma}
\begin{proof}
Since $\{\v^k\} \subset \Omega(\epsilon)$, the sequence $\{ z_i^k s_i^k \}$
is all bounded below and away from zero for each $i=1, \ldots, p$;
more precisely, $z_i^k s_i^k \ge  \frac{1}{2} \frac{z^0_i s_i^0}{\phi (\v^0)} 
\phi (\v^k) \ge \frac{1}{2}
\frac{z_i^0 s_i^0}{\phi(\v^0)}\epsilon$ for each $i$. Since $\{ s_i^k \}$ 
is bounded, this shows $\{ z_i^k \}$ is bounded below and away from zero. 
Similarly, $\{ s_i^k \}$ is also bounded below and away from zero.
In view of Remark \ref{assumption0}, (B2) implies $\nabla \h(\x^k)$ is 
linearly independent. From (\ref{dlagrangian}), it follows
\[
\y^k = -((\nabla \h(\x^k))^{\T}\nabla \h(\x^k))^{-1}(\nabla \h(\x^k))^{\T}
\left[ \nabla_{\x} L(\v^k) - \nabla f(\x^k)+ \nabla \g(\x^k) \w^k
\right],
\]
hence, $\{\y^k\}$ is bounded because of the following claims: 
(a) $\{ \w^k\}$ is bounded, (b) (B1) implies $\{ \nabla f(\x^k) \}$ 
and $\{ \nabla \g(\x^k) \}$ are bounded, 
(c) Lemma \ref{kvBound} indicates $\{ \nabla \h(\x^k) \}$ and 
$\{ \nabla_{\x} L(\v^k) \}$ are bounded, and (d) (B2) implies
$((\nabla \h(\x^k))^{\T}\nabla \h(\x^k))^{-1}$ exists and is bounded.
\hfill \qed
\end{proof}

Summarizing Lemmas \ref{kvBound}, \ref{sBound}, \ref{zBound}, and \ref{yBound}, 
we have the following Lemma.

\begin{lemma}\label{boundAboveBelow}
Assume that (B1) and (B2) hold, and $\{ \phi(\v^0) \}$ is bounded. 
For some fixed $\epsilon > 0$, if the sequence 
$\{\v^k\}$ satisfies  $\{\v^k\} \subset \Omega(\epsilon)$, then  $\{ \v^k \}$ 
is bounded and $\{(\w^k, \s^k, \z^k)\} > \0$
is bounded below and away from zero.
\label{boundedness}
\end{lemma}

Since $\S^k$ and $\Z^k$ are bounded below and away from zero, we would like
introduce two more assumptions.

\vspace{0.05in}
\noindent
{\bf Assumptions}
\begin{itemize}
\item[(B3)] The matrix $\nabla_{\x}^2 L(\v)+\nabla \g(\x) \S^{-1} \Z (\nabla \g(\x))^{\T}$ 
is invertible for any $\v$ in $\Omega(\epsilon)$,
and  the set $\{ (\nabla_{\x}^2 L(\v)+\nabla \g(\x) \S^{-1} \Z (\nabla \g(\x))^{\T})^{-1} :
\v \in \Omega(\epsilon)\}$ is bounded.
\item[(B4)]  For each k, the function $\{ \phi'(\x^k(\alpha)) \}$ is Lipschitz continuous, i.e.,
$\phi'(\x^k(\alpha)) -\phi'(\x^k) \le L \alpha$ for a constant $L$.
\end{itemize}

\begin{remark}
Assumption (B3) is necessary and explains why an extensively studied problem \cite{wb00} is not globally
convergent (see problem (\ref{LD1})).
\end{remark}

The following lemma will be used to show the boundedness 
of the inverse of the Jacobian $\{ \k'(\v^k)\}$.

\begin{lemma}[\cite{ls02}]\label{blockInverse}
Let ${\bf R}$ be a block matrix
\[
{\bf R}=\left[ \begin{array}{cc}
\A & \B \\ \C & \D  \end{array}   \right].
\]
If $\A$ and $\D-\C\A^{-1}\B$ are invertible,
or $\D$ and $\A-\B\D^{-1}\C$ are invertible, then ${\bf R}$ is invertible.
\end{lemma}

The next lemma is given in \cite{yiy22}, we provide the proof here so that the readers 
can check that the result is applicable to all three algorithms listed in the previous section.

\begin{lemma}[\cite{yiy22}]\label{lem:inverseFprime}
Assume that (B1)-(B3) hold. If $\{\v^k\} \subset \Omega(\epsilon)$ for a fixed $\epsilon > 0$, 
then $\{ [\k'(\v^k)]^{-1} \}$ is bounded.
\label{invertibility}
\label{FpInv}
\end{lemma}
\begin{proof}
Let 
\begin{align*}
& \A^k = \left[ \begin{array}{cc}
\nabla_{\x}^2 L(\v^k)   &  -\nabla \h(\x^k) \\
(\nabla \h(\x^k))^\T  & \0
\end{array}  \right], 
\B^k = \left[ \begin{array}{ccc}
-\nabla \g(\x^k)  &  \0 & \0  \\
\0  &  \0 &  \0 \\
\end{array}  \right], \\
& \C^k = \left[ \begin{array}{cc}
(\nabla \g(\x^k))^\T    &  \0  \\
\0  &  \0  \\
\0  &    \0 
\end{array}  \right], \ \text{and} \
\D^k = \left[ \begin{array}{ccc}
\0  & -\I  & \0  \\
\I &  0  & -\I  \\
\0 &   \Z^k  & \S^k   
\end{array}  \right].
\end{align*}
Then, we have
\begin{align}
\k' (\v^k) = \left[ \begin{array}{ccccc}
\nabla_{\x}^2 L(\v^k) & -\nabla \h(\x^k) & -\nabla \g(\x^k) & \0 & \0 \\
\left( \nabla \h(\x^k)\right)^{\T} & \0 & \0 & \0 & \0  \\
\left( \nabla \g(\x^k)\right)^{\T} & \0 & \0 & -\I & \0 \\
\0  & \0  &  \I  &  \0  &  -\I  \\
\0 & \0 & \0  &  \Z^k & \S^k
\end{array} \right]
= \left[ \begin{array}{cc} \A^k & \B^k \\ \C^k & \D^k \end{array}\right]
\end{align}
From Lemma \ref{boundAboveBelow},  the two sequences $\{\s^k\}$ and $\{\z^k\}$ are
bounded and each component of the two sequences are bounded below and away from zeros, 
therefore, the sequence $\{(\D^k)^{-1}\}$ is also bounded, where
\[
(\D^k)^{-1} = \left[ \begin{array}{ccc}
(\S^k)^{-1}\Z^k &  \I  & (\S^k)^{-1}  \\
-\I &  \0  & \0  \\
(\S^k)^{-1}\Z^k &  \0  &  (\S^k)^{-1}
\end{array}  \right].
\]
Let 
$\bar{\L} = \nabla_{\x}^2 L(\v^k) + \nabla \g(\x^k) (\S^k)^{-1} \Z^k \nabla \g(\x^k)^{\T}$.
We know that $\bar{\L}$ is invertible from Lemma~\ref{boundAboveBelow} and (B3),
therefore,  
$(\nabla \h(\x^k))^{\T} \bar{\L}^{-1} \nabla \h(\x^k)$ is also invertible from (B2).
Therefore, 
\[
\H^k :=  \A^k - \B^k(\D^k)^{-1}\C^k   = \left[ \begin{array}{cc}
\bar{\L}  & -\nabla \h(\x^k) \\
(\nabla \h(\x^k))^{\T}  &  \0
\end{array}  \right]
\]
is invertible from Lemma \ref{blockInverse}. Since 
$\D^k$ and $\H^k$ are invertible, using Lemma \ref{blockInverse} again,
we conclude that $\k'(\v^k)$ is invertible. 

Next, we show the boundedness of $\{[\k'(\v^k)]^{-1}\}$.
Since $[\k'(\v^k)]^{-1}$ is given by
\[
[\k'(\v^k)]^{-1} = \left[\begin{array}{cc}
(\H^k)^{-1} & -(\H^k)^{-1} \B^k (\D^k)^{-1} \\ 
-(\D^k)^{-1} \C^k (\H^k)^{-1} & (\D^k)^{-1} \C^k (\H^k)^{-1} \B^k (\D^k)^{-1} + (\D^k)^{-1}
\end{array}\right],
\]
we need to show that $\{(\H^k)^{-1}\}$ is bounded.
For each $k$, $(\H^k)^{-1}$ is given as follows:
\[
(\H^k)^{-1} = \left[\begin{array}{cc}
\bar{\L}^{-1} + \bar{\L}^{-1} \nabla \h(\x^k) \bar{\H}^{-1} (\nabla \h(\x^k))^{\T} 
\bar{\L}^{-1} &  \bar{\L}^{-1} \nabla \h(\x^k) \bar{\H}^{-1} \\
\bar{\H}^{-1} (\nabla \h(\x^k))^{\T} \bar{\L}^{-1} & + \bar{\H}^{-1}
\end{array}\right],
\]
$\bar{\H} = (\nabla \h(\x^k))^{\T} \bar{\L}^{-1} \nabla \h(\x^k)$.
Therefore, it is enough to show 
the boundedness of $\bar{\L}, \bar{\H}, \bar{\L}^{-1}$ and $\bar{\H}^{-1}$,
and this is done by Assumptions (B3) and (B2), and Lemma~\ref{boundAboveBelow}.
This completes the proof. 
\hfill\qed
\end{proof}

The next lemma follows directly from Lemma~\ref{invertibility}.

\begin{lemma}\label{ettzTheorem}
Assume that (B1)-(B3) hold and $\{\v^k\} \subset \Omega(\epsilon)$, then 
(a) Lines 7 and 9 in Algorithms \ref{mainAlgo0}, \ref{mainAlgo2}, and \ref{mainAlgo1} 
are well-defined, and (b) the sequences $\{ \dot{\v}^k \}$ and $\{ \ddot{\v}^k \}$ 
are bounded.
\end{lemma}
\begin{proof}
The claim (a) follows directly from Lemma~\ref{invertibility}.
In view of (\ref{firstOrderM}), the boundedness of 
$\{ [\k'(\v^k)]^{-1} \}$ and $\{ \v^k \}$ guarantees the boundedness of 
$\{ \dot{\v}^k \}$. Using (\ref{secondOrder}) or (\ref{secondOrderM}), 
plus the boundedness of $\{ \dot{\v}^k \}$,
the boundedness of $\{ \ddot{\v}^k \}$ follows from a similar argument. 
This shows claim (b).
\hfill\qed
\end{proof}

The boundedness results proved in the above lemmas guarantee
that the step size $\alpha_k$ is bounded below and away from zero 
for every $k$ before the algorithms converge. The following lemmas 
are aimed at showing this claim.

\begin{lemma}
Let $\zeta_0 > 0$ be 
the minimum value of the elements in $\{ (\w^k, \s^k, \z^k) \}$, and 
$(\zeta_1, \zeta_2) > \0$ be the upper bound of the absolute values of the 
element  in $\{ (\dot{\w}^k, \dot{\s}^k, \dot{\z}^k) \}$
and $\{ (\ddot{\w}^k, \ddot{\s}^k, \ddot{\z}^k) \}$. Assume that (B1)-(B3) hold. 
If $\{\v^k\} \subset \Omega(\epsilon)$, then the sequence
$\{ \tilde{\alpha}_k \}$ is bounded below and away from zero.
More specifically,
$\tilde{\alpha}_k \ge \frac{\zeta_1 + \sqrt{\zeta_1^2 + 4 (1-\delta_1) \zeta_0 \zeta_2}}{2\zeta_2}$ 
for all $k \ge 0$. 
\label{tildeAlpha}
\end{lemma}
\begin{proof}
We can rewrite (\ref{positive}) as
\begin{equation}
(1 - \delta_1) \w^k -\dot{\w}^k\sin(\alpha)+\ddot{\w}^k(1-\cos(\alpha))
\ge 0.
\label{alphaiNew}
\end{equation}
In view of Lemma \ref{boundAboveBelow}, $\{(\w^k, \s^k, \z^k)\} >\0$ 
is bounded below and away from zero, thus
$\{(1 - \delta_1) \w^k\} $ is bounded below and away from zero.
Since $\{\dot{\w}^k\}$ and $\{\ddot{\w}^k\}$ are bounded from Lemma~\ref{ettzTheorem}, 
$\{\tilde{\alpha}_k\}$ should be bounded below and away from zero such that the
inequality (\ref{alphaiNew}) holds for all $\alpha \in [0, \tilde{\alpha}_k]$. 
More precisely, since we take $\alpha$ from the range 
$\left[0, \frac{\pi}{2}\right]$, it holds that $\sin(\alpha) \le \alpha$ 
and $1-\cos(\alpha) \le 1-\cos^2(\alpha) =\sin^2(\alpha) \le \alpha^2$.
Therefore, (\ref{alphaiNew}) holds when
\begin{align*}
(1 - \delta_1) \zeta_0 -\zeta_1 \sin(\alpha) - \zeta_2 (1-\cos(\alpha))
\ge (1 - \delta_1) \zeta_0 - \zeta_1 \alpha - \zeta_2 \alpha^2 \ge 0.
\end{align*}
This indicates that 
(\ref{alphaiNew}) holds for any $\alpha \in 
\left[0, \frac{\zeta_1 + \sqrt{\zeta_1^2 + 4 (1-\delta_1) \zeta_0 \zeta_2}}{2\zeta_2} \right]$.
We can apply the same arguments
to $\{\s^k \}$ and $\{\z^k \}$. 
This proves the lemma.
\hfill \qed
\end{proof}

Next, we show that  $\{ \bar{\alpha}_k \}$ is bounded below and away from zero.
\begin{lemma} 
Assume that $\{ \v^k\} \subset \Omega(\epsilon)$, and Assumptions (B1)-(B3) hold. 
Then the sequence $\{ \bar{\alpha}_k \}$ is bounded below and away from zero. 
\label{barAlpha}
\end{lemma}
\begin{proof}
According to Lemma \ref{boundedness}, $\{ \s^k\}$ is bounded below and away from zero, 
it follows that $\g(\x^k) \ge \s^k$ is bounded below and away from zero
for all $k$ before the algorithms converge. Since $\g(\x)$ is continuous, and 
according to Lemma \ref{ettzTheorem}, $\{ \dot{\v}^k \}$ and $\{ \ddot{\v}^k \}$ 
are bounded, it must have an $\bar{\alpha}_k >0$ bounded below and away from zero
such that for all $\alpha \in (0, \bar{\alpha}_k]$, 
$\g(\x(\alpha))=\g(\x^k-\dot{\x}\sin(\alpha)+\ddot{\x}(1-\cos(\alpha)))> 0$.
\hfill \qed
\end{proof}

Then, we show that  $\{ \check{\alpha}_k \}$ is bounded below and away from zero.

\begin{lemma} 
If $\{ \v^k\} \subset \Omega(\epsilon)$ and Assumptions (B1)-(B4) hold, then the sequence
$\{ \check{\alpha}_k \}$ is bounded below and away from zero. More precisely, let 
$\epsilon>0$, $\sigma \ge \bar{\sigma} \ge 0$, and $\delta_2$ satisfy
$\frac{\phi'({0})^2}{2L} \ge \frac{2\epsilon^2(1-\sigma)^2}{L} \ge \delta_2>0$. 
Then, Condition (\ref{alphaCheck}) will hold for 
$\check{\alpha}_k \ge \frac{2\epsilon (1-\sigma)}{L}$.
\label{checkAlpha}
\end{lemma}
\begin{proof}
Using (\ref{vAlpha}), for the sake of brevity, we will use $\phi(\alpha)$ for $\phi(\v(\alpha))$ and
use $\phi(0)$ for $\phi(\v^k)$, respectively. Using the mean-value theorem, Assumption (B4),
and (\ref{tmp0}), we have
\begin{eqnarray}
\phi(\alpha) & = & \phi(0) + \int_0^1 \phi'(t\alpha) dt \alpha 
\nonumber \\
 & = &  \phi(0) + \int_0^1 [\phi'(t\alpha) -\phi'(0) + \phi'(0)] dt \alpha
\nonumber \\
 & = &  \phi(0) + \phi'(0) \alpha t|_0^1 +\int_0^1 [\phi'(t\alpha) -\phi'(0)] dt \alpha
\nonumber \\
 & \le &  \phi(0) + \phi'(0) \alpha +\int_0^1 L \alpha^2 t dt
\nonumber \\
  & = &  \phi(0) + \phi'(0) \alpha + \frac{1}{2} L\alpha^2 t^2 |_{t=0}^1
\nonumber \\
 & = &  \phi(0) + \phi'(0) \alpha + \frac{1}{2} L\alpha^2.
\label{ckAlpha}
\end{eqnarray}
In order for $\phi(\alpha) \le  \phi(0) -\delta_2$, it may require that
$\phi'(0) \alpha + \frac{1}{2} L\alpha^2 \le -\delta_2$, which will hold for 
${\alpha} \in \left[ \frac{-\phi'(0)-\sqrt{\phi'({0})^2-2L\delta_2}}{L},
\frac{-\phi'(0)+\sqrt{\phi'({0})^2-2L\delta_2}}{L}\right]$. Multiplying $-1$
on both sides of (\ref{objReduction}) yields 
$-\phi'(0) \ge 2\phi(0)(1-\sigma) = 2\phi(\v^k)(1-\sigma) \ge 2\epsilon(1-\sigma)$. 
Therefore, assuming that we have selected $\delta_2 \le \frac{\phi'({0})^2}{2L}$, 
we have
\begin{eqnarray}
\check{\alpha}_k & \ge & \frac{-\phi'(0)+\sqrt{\phi'({0})^2-2L\delta_2}}{L}
\nonumber \\
& \ge &  \frac{-\phi'(0)}{L}
\nonumber \\
& \ge &  \frac{2\phi(\v^k)(1-\sigma)}{L}
\nonumber \\
 & \ge &  \frac{2\epsilon (1-\sigma)}{L}
\end{eqnarray}
is bounded below and away from zero because $(1-\sigma)$
is bounded below and away from zero, and $L$ is bounded.
\hfill \qed
\end{proof}

The following Lemma is used to show in Lemma~\ref{AlphaH} that $\{ \hat{\alpha}_k \}$ 
is bounded below and away from zero.

\begin{lemma}[\cite{yang17}]
Assume that ${\v^k}$ is the current point, 
$\dot{\v}$ satisfies (\ref{firstOrder}), and $\ddot{\v}$ satisfies (\ref{secondOrder}) 
or  (\ref{secondOrderM}). 
Let $\v(\alpha)$ be computed by using (\ref{vAlpha}). Then,
\begin{eqnarray}
{z}_i(\alpha) {s}_i(\alpha) & = &  {z}_i  {s}_i (1-\sin(\alpha)) +\sigma  {\mu} \sin(\alpha)
-(\dot{z}_i\ddot{s}_i+\ddot{z}_i\dot{s}_i)\sin(\alpha)(1-\cos(\alpha))
\nonumber \\
& &  +(\ddot{z}_i\ddot{s}_i-\dot{z}_i\dot{s}_i) (1-\cos(\alpha))^2.
\label{compWise}
\end{eqnarray}
\end{lemma}
\begin{proof}
Using the last rows of (\ref{firstOrderM}) and (\ref{secondOrder}), 
we have
\begin{eqnarray}
{z}_i\AN{\alpha}{s}_i\AN{\alpha} & = & 
[{z}_i-\dot{z}_i\sin(\alpha)+\ddot{z}_i(1-\cos(\alpha))]
[{s}_i-\dot{s}_i\sin(\alpha)+\ddot{s}_i(1-\cos(\alpha))]
\nonumber \\
& = & {z}_i {s}_i-(\dot{z}_i{s}_i+ {z}_i\dot{s}_i)\sin(\alpha)
+(\ddot{z}_i {s}_i+ {z}_i\ddot{s}_i) (1-\cos(\alpha)) 
+ \dot{z}_i\dot{s}_i \sin^2(\alpha) \nonumber \\
& & -(\dot{z}_i\ddot{s}_i+\ddot{z}_i\dot{s}_i)
\sin(\alpha)(1-\cos(\alpha))
+\ddot{z}_i\ddot{s}_i  (1-\cos(\alpha))^2
\nonumber \\
& = & {z}_i {s}_i (1-\sin(\alpha)) +\sigma {\mu} \sin(\alpha) -2\dot{z}_i\dot{s}_i
(1-\cos(\alpha)) + \dot{z}_i\dot{s}_i \sin^2(\alpha) \nonumber \\
& & -(\dot{z}_i\ddot{s}_i+\ddot{z}_i\dot{s}_i)
\sin(\alpha)(1-\cos(\alpha))
+\ddot{z}_i\ddot{s}_i (1-\cos(\alpha))^2
\nonumber \\
& = & {z}_i {s}_i (1-\sin(\alpha)) +\sigma {\mu} \sin(\alpha)
+\dot{z}_i\dot{s}_i (\sin^2(\alpha)+2\cos(\alpha)-2)
\nonumber \\
& & -(\dot{z}_i\ddot{s}_i+\ddot{z}_i\dot{s}_i)\sin(\alpha)(1-\cos(\alpha))
+\ddot{z}_i\ddot{s}_i (1-\cos(\alpha))^2.
\nonumber 
\end{eqnarray}
Substituting $\sin^2(\alpha)+2\cos(\alpha)-2=-1+2\cos(\alpha)
-\cos^2(\alpha)=-(1-\cos(\alpha))^2$ into the last equation gives
(\ref{compWise}). Since the last rows of (\ref{secondOrder}) and (\ref{secondOrderM})
are the same, this conclusion applies to Algorithms \ref{mainAlgo0}, 
\ref{mainAlgo2}, and \ref{mainAlgo1}.
\hfill \qed 
\end{proof}

\begin{lemma}[\cite{yiy22}]
Assume that (B1)-(B4) hold. If $\{ \v^k\} \subset \Omega(\epsilon)$ 
for some $\epsilon > 0$, then $\{\hat{\alpha}_k\}$ is bounded below and away from zero.
\label{AlphaH}
\end{lemma}

\begin{proof}  
For each $k$, find $i$ such that  $z_i^k s_i^k = \min(\Z^k \s^k)$, and let 
$\eta_1^k=  \dot{z}_i^k\ddot{s}_i^k+\ddot{z}_i^k\dot{s}_i^k$ and
$\eta_2^k= \ddot{z}_i^k\ddot{s}_i^k-\dot{z}_i^k\dot{s}_i^k$.
Since $\{\dot{\v}^k\}$ and $\{\ddot{\v}^k\}$ are bounded
according to Lemma~\ref{ettzTheorem},
the sequences $\{ | \eta_1^k | \}$ and $\{ | \eta_2^k | \}$
are also bounded. The proof is based on induction.
For $k=1$, from (\ref{conC1}) and (\ref{compWise}), we have 
\begin{eqnarray}
& & \min (\Z^1  \s^1) - \frac{1}{2} \min (\Z^0 \s^0) \frac{\phi (\v^1)}{\phi (\v^0)}  \nonumber \\
& \ge & z_i^1 s_i^1 - \frac{1}{2} \min (\Z^0 \s^0) 
[1- 2\rho (1-\sigma_0) \sin(\alpha_0)]  \nonumber \\
& \ge &  {z}_i^0  {s}_i^0 (1-\sin(\alpha_0)) +\sigma_0  {\mu_0} \sin(\alpha_0)
-\eta_1^0\sin(\alpha_0)(1-\cos(\alpha_0))
+\eta_2^0 (1-\cos(\alpha_0))^2  \nonumber \\
& & -\frac{1}{2} (z_i^0 s_i^0) [1- 2\rho (1-\sigma_0) \sin(\alpha_0)] \nonumber \\
& = & \frac{1}{2} {z}_i^0  {s}_i^0 - {z}_i^0  {s}_i^0\sin(\alpha_0)
+\sigma_0  {\mu_0} \sin(\alpha_0) -\eta_1^0\sin(\alpha_0)(1-\cos(\alpha_0)) \nonumber \\
& & +\eta_2^0 (1-\cos(\alpha_0))^2
+{z}_i^0  {s}_i^0 \rho (1-\sigma_0) \sin(\alpha_0).
\label{inter1}
\end{eqnarray}

Since $\v^k \in \Omega(\epsilon)$, we know 
$z_i^k s_i^k \ge \frac{1}{2} \min(\Z^0  \s^0)
\frac{\phi(\v^k)}{\phi(\v^0)} \ge \frac{1}{2} \min(\Z^0  s^0)
\frac{\epsilon}{\phi(\v^0)} > 0$, therefore
there must be $\alpha_0 > 0 $ such that
the last express in (\ref{inter1}) is greater than zero.
Next, for $k>1$, assume that there exists $\alpha_{k-1} >0$ such that 
\begin{equation}
\min (\Z^k  \s^k) - \frac{1}{2} \min (\Z^0  \s^0) 
\frac{\phi (\v^k)}{\phi (\v^0)} > 0,
\label{inter2}
\end{equation}
then we show that there exists $\alpha_{k}>0$ such that 
\begin{equation}
\min (\Z^{k+1} \s^{k+1}) - \frac{1}{2} \min (\Z^0  \s^0) 
\frac{\phi (\v^{k+1})}{\phi (\v^0)} > 0.
\nonumber
\end{equation}

From (\ref{conC1}) and (\ref{compWise}), it follows 
\begin{eqnarray}
& & \min (\Z^{k+1} \s^{k+1}) - \frac{1}{2} \min (\Z^0 \s^0) 
\frac{\phi (\v^{k+1})}{\phi (\v^0)}  \nonumber \\
& \ge & z_i^{k+1} s_i^{k+1} - \frac{1}{2} \min (\Z^0 \s^0) 
\frac{\phi (\v^k)}{\phi (\v^0)}[1- 2\rho (1-\sigma_k) \sin(\alpha_k)]  
\nonumber \\
& \ge &  {z}_i^k  {s}_i^k (1-\sin(\alpha_k)) +\sigma_k  {\mu_k} \sin(\alpha_k)
-\eta_1^k\sin(\alpha_k)(1-\cos(\alpha_k))
+\eta_2^k (1-\cos(\alpha_k))^2  \nonumber \\
& & -\frac{1}{2} \min (\Z^0 \s^0)  
\frac{\phi (\v^k)}{\phi (\v^0)}[1- 2\rho (1-\sigma_k) \sin(\alpha_k)] 
\label{inter3}
\end{eqnarray}
Since ${z}_i^k  {s}_i^k \ge \min (\Z^k \s^k) \ge \frac{1}{2} \min(\Z^0 \s^0)
\frac{\epsilon}{\phi(\v^0)} > 0 $ and (\ref{inter2}), 
we can find $\alpha_k >0$ such
that the last express in (\ref{inter3}) is greater than zero.

We know that  $\{z_i^k s_i^k\}$ is bounded below and away from zero, 
$\{\sigma_k\} \ge 0$, and $\{ |\eta_1^k|\}$ and $\{|\eta_2^k|\}$ are bounded 
due to Lemma~\ref{ettzTheorem}. Therefore, $\{\hat{\alpha}_k\}$  
is bounded below and away from zero.
\hfill \qed
\end{proof}

From Lemmas~\ref{tildeAlpha}, \ref{barAlpha}, \ref{checkAlpha}, \ref{AlphaH}, we immediately
have the following result.

\begin{lemma}
Assume (B1)-(B4) hold, then for Algorithms \ref{mainAlgo0}, \ref{mainAlgo2}, and
\ref{mainAlgo1}, $\{\alpha_k\}$ is bounded below and away from zero.
\label{alphaKS}
\end{lemma}
\begin{proof}
For Algorithms \ref{mainAlgo0}, \ref{mainAlgo2}, and  \ref{mainAlgo1}, 
since $\tilde{\alpha}_k$, $\bar{\alpha}_k$, $\check{\alpha}_k$, and $ \hat{\alpha}_k$ 
are all bounded below and away from zero,
$\alpha_k=\min\{ \tilde{\alpha}_k,\bar{\alpha}_k, \check{\alpha}_k, \hat{\alpha}_k \}$ 
is bounded below and away from zero. This completes the proof.
\hfill \qed
\label{alphaK3}
\end{proof}

We are now ready to prove the convergence for all three Algorithms.

\begin{theorem}\label{global} 	
Assume (B1)-(B4) hold. Then, for Algorithms \ref{mainAlgo0}, \ref{mainAlgo2}, and
\ref{mainAlgo1}, we have (i) for all $k \ge 0$, the sequence $\{\phi(\v^k)\}$ 
decreases faster than a constant rate, and (ii) all three algorithms terminate in finite 
iterations.
\end{theorem}
\begin{proof}
Since $\{\alpha_k\}$ is bounded below and away from zero,
there must be an ${\alpha^*} > 0$ such that 
$\alpha_k \ge \alpha^* = \min \{  \tilde{\alpha}_k,\bar{\alpha}_k, \check{\alpha}_k, \hat{\alpha}_k   \} > 0$
for Algorithm~\ref{mainAlgo0}, ~\ref{mainAlgo2}, and \ref{mainAlgo1}.
Since $\alpha^*$ is bounded below and away from zeros, this shows that 
$\{\phi(\v^k)\}$ decreases faster than a constant rate due to \eqref{conC1}.
Since $\phi(\v^0)$ is bounded, and $\{\phi(\v^k)\}$ decreases faster than
a constant rate, it needs only finite iterations $K$ to 
find $\phi(\v^K) \le \epsilon$ with $(\w^K, \s^K, \z^K) \in \R_+^{3p}$.
\hfill \qed
\end{proof}

\section{Preliminary test results}
\label{sec:test}

The benefit of using arc-search can easily be seen from Problem 19 (HS19) in 
\cite{hs81} which is given as follows:
\begin{align}
\begin{array}{rcl}
\min &:& (x_1 -10)^3 + (x_2 -20)^3 \\
\textrm{s.t.} &:& (x_1 -5)^2 +(x_2-5)^2 - 100 \ge 0, \\
 & &  -(x_2 - 5)^2 -(x_1-6)^2 +82.81 \ge 0, \label{exampleNP}\\
& & 13 \le x_1 \le 100, \\ 
& &  0 \le x_2  \le 100.
\end{array}
\end{align}
Figure 1(a) shows the real constraint set of HS19 which is between 
two curves in red lines. The contour lines describe the level of the objective function. 
The optimal solution is at the intersection of the two red curves marked by 'x'. 
Figure 1(b) shows that the ellipse constructed by using
(\ref{vAlpha}) passes the current iterate marked by a red circle 'o'. 
Clearly, searching along the ellipse is better than searching any straight 
line for this nonlinear constraint problem. Figure 1(c) 
shows the iterates that approach to the optimal solution.

\begin{figure}
    \centering
    \includegraphics[width=0.3\textwidth]{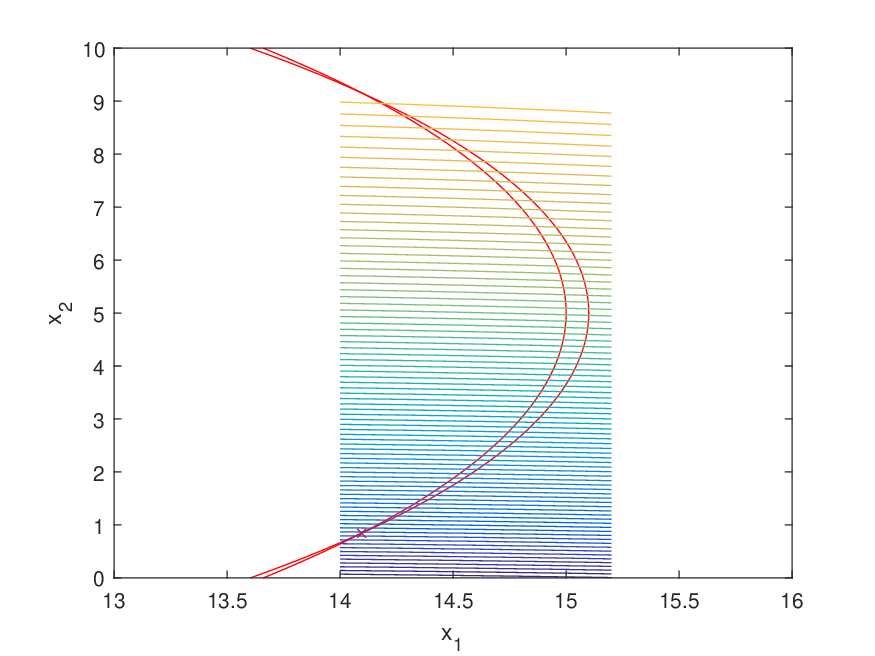}
    \includegraphics[width=0.3\textwidth]{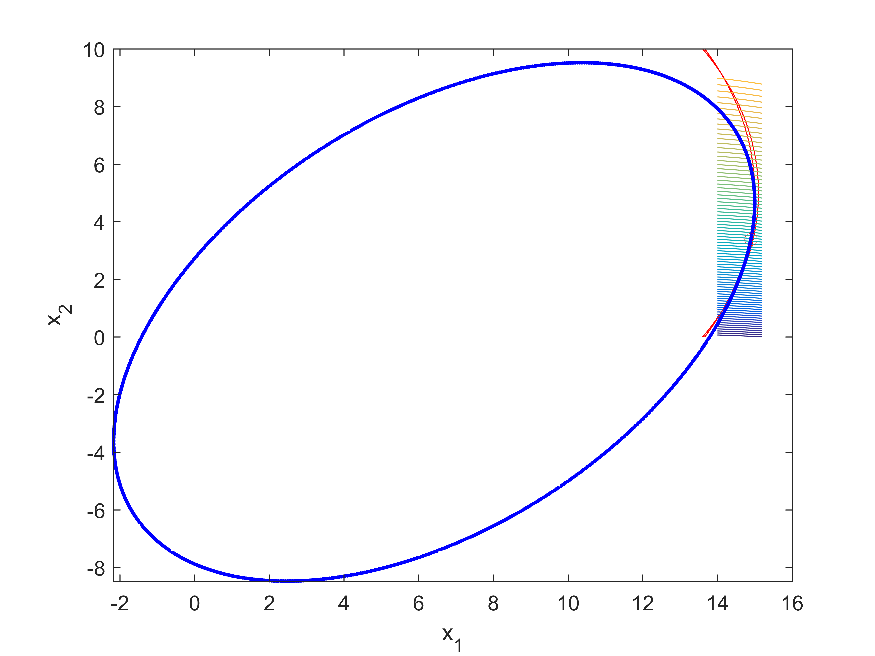}
    \includegraphics[width=0.3\textwidth]{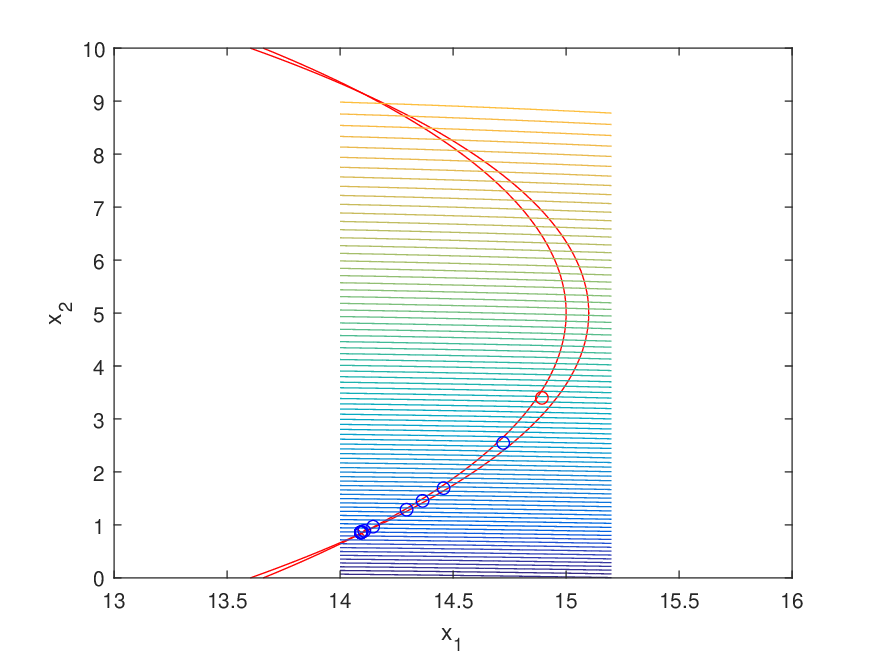}
    \caption{(a) The constraint set of problem HS19. 
(b) Ellipse approximation of the central-path of problem HS19. 
(c) Ellipse iterates of the central-path of problem HS19.}
    \label{HS19}
\end{figure}

%
%

\subsection{Two simple but hard problems}

In \cite{ld20}, two simple and hard problems are tested. To demonstrate the
efficiency of the proposed algorithm, we tested the algorithm against these problems.
The first one was extensively studied in \cite{wb00}.
\begin{align}
\begin{array}{rcl}
\min_{x_1, x_2, x_3} &:& x_1  \\
\textrm{s.t.} &:& x_1^2 -x_2 -1 =0, \label{LD1}  \\
 & & x_1-x_3 -2 =0, \\
& &  x_2  \ge 0, \hspace{0.1in} x_3  \ge 0.
\end{array}
\end{align}
It is showed in \cite{wb00} that many interior-point algorims using line search 
fail to solve this problem because the iterates is confined to a region of 
an inconsistent system of equations and inequalities. It is easy to check that 
Condition (B3) is violated for this problem. However, a simple implementation
trick (which will be discussed in a follow-on paper) will get rid of the difficulty.
As a matter of fact, the algorithm converges to the optimal solution $(2,3,0)$
in $39$ iterations from the standard initial condition (-4, 1,1 ).

The second example is a standard test problem taken from \cite[Problem 13]{hs81}:
\begin{align}
\begin{array}{rcl}
\min_{x_1, x_2} &:& (x_1-2)^2+x_2^2  \\
\textrm{s.t.} &:& (1-x_1)^3 - x_2 \ge 0, \label{LD2} \\
& &  x_1  \ge 0, \hspace{0.1in} x_2  \ge 0.
\end{array}
\end{align}
This problem was not solved in \cite{sv00,yamashita98}. Starting from the standard
initial point (-2,-2), the proposed algorithm terminates at (0.9997,0) after 25 iterations, which is 
closer to the optimal solution (1,0) than the result of (0.9905,-0.0000) given in \cite{ld20}.

\subsection{Test on some unsolvable problems by \cite{yiy22}}

\begin{table}
\tiny
	\centering
	\caption{Test results for a subset of problems in \cite{hs81}}
	\begin{tabular}{|c||c|c|c|c|} \hline
& \multicolumn{4}{|c|}{Agorithm \ref{mainAlgo1} } \\ \hline
Prob & Obj & Iter & Seconds & Conv $\phi$  \\ \hline
16  & 0.25 & 22 & 0.231982 & 1.2313e-15 \\ \hline
17  & 1 & 22 & 0.229433 & 1.3279e-15 \\ \hline
19   & -6961.8139 & 22 & 0.206065 & 7.563e-11\\ \hline
23  & 2 & 22 & 0.253886 & 3.0712e-13 \\ \hline
32   & 1 & 22 & 0.273155 & 7.6672e-16 \\ \hline
59  & -7.8028 & 24 & 0.253580 & 2.4705e-12 \\ \hline
64   & 6299.8424 &19 & 0.167498 & 1.139e-10  \\ \hline
66    & 0.51816 & 22 & 0.275467 & 1.5841e-12 \\ \hline
71   & 17.014 & 37 & 0.576570  &  3.1672e-13 \\ \hline
80   & 0.05395  & 20 & 0.432452 &  5.7296e-15 \\ \hline
84  & -5280335.2971 & 27 & 0.689569 & 6.1572e-05 \\ \hline
95   & 0.015621  & 23 & 0.656584 & 4.4154e-13 \\ \hline
96   & 0.015621 & 20 & 0.520563 & 3.6668e-13 \\ \hline
97   & 4.6451 & 25 & 0.662390 & 2.6823e-11  \\ \hline
98    & 4.6451 & 26 & 0.793426 & 6.9447e-12  \\ \hline
101   & 1809.7648 & 53  & 10.802482  &  1.5096e-09  \\  \hline
108   & -0.86603 & 22 & 1.674474 & 1.1108e-15 \\ \hline
	\end{tabular}
\label{HStest}
\normalsize
\end{table}

We tested Algorithm \ref{mainAlgo1} on a subset of widely used benchmark 
problems (see Table \ref{HStest}) in \cite{hs81}. It is worthwhile to note that the 
algorithm proposed in \cite{yiy22} failed to solve this subset of problems. 
But Algorithm \ref{mainAlgo1} solves all these problems. The test result for
Algorithm~\ref{mainAlgo1} is summarized in Table \ref{HStest}.

The preliminary test result for these benchmark problems shows that (a) the newly developed 
algorithms are likely more robust than the one in \cite{yiy22}, and (b) it is very important to 
appropriately handle the ill-conditioned matrix (\ref{firstOrder}) (this will be discussed in
detail in a follow-on paper that focuses on the implementation and 
will report more extensive test results).

The proposed algorithm is also tested for some spacecraft trajectory optimization problems
\cite{yph24} and solves a spacecraft formation flying orbit design problem \cite{yang25}.

\section{Conclusions}
\label{sec:conclusions}

In this paper, we proposed an arc-search interior-point algorithm for the nonlinear
optimization problem. The algorithm searches the optimizer along an arc that is part of
an ellipse instead of a traditional straight line. Since the arc stays in the interior set
longer than any straight line, the arc step is longer than a line step, thereby producing
a better iterate and improving the efficiency of the algorithm. Convergence of the
algorithm is established under some mild conditions. Preliminary test results show that
the algorithm is very promising. A follow-on paper is under preparation to provide 
details of the algorithm implementation and to report extensive test result.

\section{Acknowledgment}

This research is partially supported by NASA’s IRAD 2023 fund SSMX22023D. 
The author is grateful to the anonymous reviewers and Dr. Robert Pritchett at 
Goddard Space Fight Center of NASA for their valuable comments that helped 
the author to improve the presentation of the paper.


\section{Declarations}
%
%
\subsection{Conflict of interests}
Author declares no conflict of interests.
%
%
%
\subsection{Data availability}

The datasets generated during and/or analysed during the current study
are available from the corresponding author on reasonable request.

\end{document}